\documentclass[jap,press]{apt}
\newif\ifjournal
\journalfalse
\ifjournal\usepackage{lineno}\linenumbers\fi
\usepackage[colorlinks=true,linkcolor=blue,urlcolor=blue,citecolor=blue,anchorcolor=blue]{hyperref}

\authornames{P. Chaudhary and K. Khandelwal}
\shorttitle{Sums of two dependent uniform random variables}

\usepackage{mathtools}
\usepackage{cite}
\usepackage{enumitem}
\numberwithin{equation}{section}

\renewcommand{\crnotice}{}
\renewcommand{\ams}[2]{\begin{center}\begin{minipage}{315.83pt}\small
  \noindent 2020 Mathematics Subject Classification:~Primary {#1}\\
  \phantom{2020 Mathematics Subject Classification:~}Secondary {#2}
  \end{minipage}\end{center}\par\normalsize}

\usepackage[capitalize,nameinlink]{cleveref}
\crefname{thm}{Theorem}{Theorems}
\crefname{prop}{Proposition}{Propositions}
\crefname{lem}{Lemma}{Lemmas}
\crefname{cor}{Corollary}{Corollaries}
\crefname{ex}{Example}{Examples}
\crefname{rem}{Remark}{Remarks}
\crefname{prob}{Problem}{Problems}
\crefname{defn}{Definition}{Definitions}

\newcommand{\R}{\mathbb{R}}
\newcommand{\Z}{\mathbb{Z}}
\newcommand{\N}{\mathbb{N}}
\newcommand{\cD}{\mathcal{D}}

\newcommand{\cG}{\mathcal{G}}

\newcommand{\Ph}{\widehat{P}}
\newcommand{\DU}{\cD_{2}^{U}}
\newcommand{\frc}{\operatorname{frac}}
\newcommand{\law}{\operatorname{law}}

\newcommand{\ind}[1]{\mathbf{1}\{#1\}}

\newif\ifanon
\anonfalse

\begin{document}

\title{Sums of two dependent uniform random variables:\\ a characterization by joint tail bounds}

\authoronetwo[Independent researcher]{Prashant Chaudhary}
\authorthree[University of Notre Dame]{Kapil Khandelwal}
\addressone{Santa Clara, CA 95051, USA.}
\emailtwo{prashant@prashantchaudhary.com}
\addressthree{Department of Civil and Environmental Engineering and Earth Sciences, University of Notre Dame, Notre Dame, IN 46556, USA.}

\begin{abstract}
Which laws can the sum of two uniform random variables have when nothing is known about how
they depend on each other? For three or more uniform summands, convex order gives the answer;
for two it does not, and the question has been open. We show that the attainable laws are
exactly those that satisfy a family of explicit joint bounds on sums of tail probabilities, one
bound for each finite set of thresholds. For laws with a density at most one, convex order alone
decides. For densities at most two, convex order implies the simplest of the joint bounds, and
the constant two cannot be raised for single-threshold bounds. The number of thresholds a law
needs is set by the spacing of its atoms, not by how many atoms it has.

\end{abstract}


\ams{60E05}{60E15; 15B51; 49N15}

\section{Introduction}\label{sec:intro}

Suppose two quantities are each known to be uniform on $[0,1]$, and nothing at all is known about how they move together. What can their sum look like? The question sounds like an exercise, but it sits at the center of risk aggregation under dependence uncertainty, where one knows the marginal laws of the positions in a portfolio and has to bound the risk of the total without a model for the dependence \cite{EmbrechtsPuccettiRuschendorf2013,BernardJiangWang2014,HauserShahverdyanEmbrechts2015}. It also shows up when two dependent p-values are averaged \cite{VovkWang2020}, and in the mixability problems of \cite{WangWang2011,WangWang2016}. The question itself came out of television audience measurement, where linear television and connected television (CTV) each report their own audience distribution and the overlap between them is not observed. The two-uniform problem is the simplest version of that question, the one with equal marginals. Write $\cD_n(F_1,\dots,F_n)$ for the set of laws of $X_1+\dots+X_n$ over all random vectors with $X_i\sim F_i$, and $\cD_n^U$ for the case of $n$ standard uniform marginals.

For three or more uniforms the answer is clean. Every attainable law is dominated in convex order by the law of $nU$, and Mao et al.\ \cite{MaoWangWang2019} proved that for $n\ge3$ nothing else is needed. Two summands turned out to be harder. Convex order is still necessary, but it is no longer sufficient \cite{MaoWang2015}, and \cite{MaoWangWang2019} settled unimodal densities, two-point laws, and equidistant three-point laws, and gave a sufficient condition for densities bounded below near $1$, noting that `a complete determination of $\cD_2$ for uniform margins is still unclear.' The classical bounds of Makarov, of R\"uschendorf, and of Frank, Nelsen, and Schweizer \cite{Makarov1981,Ruschendorf1982,FrankNelsenSchweizer1987} (see \cite{ZhangRichardson2024} for a recent account) control one tail probability at a time, so they describe the worst case at each threshold but not which laws can actually occur.

This paper gives that description. The short version is that convex order has to be replaced by bounds that look at several tail probabilities at once. For thresholds $0<\eta_1\le\dots\le\eta_r<1$, put $\eta_{r+1}=1$ and let $\xi_\ell$ be the largest sum $\eta_q+\eta_{q'}$ with $q+q'=\ell+1$. We show in \cref{thm:tails} that a law of $S$ on $[0,2]$ with mean one is attainable exactly when
\[
\sum_{\ell=1}^{2r}\mathbb{P}(S\ge\xi_\ell)\ \le\ 2\sum_{q=1}^r(1-\eta_q)
\]
for every choice of thresholds, and that distinct thresholds are enough. With a single threshold this reads $\mathbb{P}(S\ge2\eta)+\mathbb{P}(S\ge1+\eta)\le2(1-\eta)$. Take $\eta=\frac12$. On its own, $\mathbb{P}(S\ge1)$ can be as large as $1$, and $\mathbb{P}(S\ge\frac32)$ as large as $\frac12$, but the bound says that the two together never exceed $1$: the two worst cases cannot happen at once. The bound is sharp for every $\eta$. Convex order follows from the family and is strictly weaker than it (\cref{ex:twopoint,ex:worked}), and the same statement holds for two copies of any law on a compact interval (\cref{prop:generalF}).

Convex order comes back for laws with a density. If the density of $S$ is at most one, convex order alone decides whether the law is attainable (\cref{thm:density}). At density two it still implies every single-step bound, every two-step bound, and every bound whose threshold gaps do not increase (\cref{prop:singlestep}, and \cref{thm:twostep,thm:gaps}). The constant two is sharp for single steps, since above it convex order fails even with its folded version added (\cref{prop:densitysharp}). Whether density two and convex order are enough in general is the main question this paper leaves open. \Cref{prob:density} states it, and \cref{prop:duality} reduces it to densities that take only the values $0$ and $2$.

The last finding runs against intuition. A law with few atoms ought to need few thresholds to rule out, but it does not: the number of thresholds is driven by the spacing of the law, not by how many atoms it has. A mean-one two-point law with spacing $\beta$, where $1/\beta$ is not an integer, needs a staircase with $\max(1,\lfloor1/\beta\rfloor)$ unit steps and is not excluded by anything smaller (\cref{thm:twopoint}), and symmetric three-point laws behave the same way (\cref{thm:threepoint}).

Three tools carry the proofs. Kantorovich duality, in a form with one potential instead of two because the marginals are equal (\cref{thm:dual}), turns membership into a family of inequalities. A fibered reduction turns the laws of $X+Y$ on any lattice $a+b\Z$ into Minkowski combinations of polytopes of doubly stochastic X-rays (\cref{thm:A,thm:B}, and \cref{prop:C}); it also reproves the two- and three-point results of \cite{MaoWangWang2019} and corrects a few slips in them (Section~\ref{sec:windows} and Appendix~\ref{app:mww}). Of the three tools, this reduction is the one most likely to be useful elsewhere, since it holds for any pair of marginals and not only for uniforms. For laws with bounded density, convex certificates (\cref{lem:densitydual}) do the work. On the grid, a short convex-minorant argument shows that convex order decides for lattice laws whose atoms are at most $1/k$ (\cref{cor:smallatoms}), and the density result follows from that lattice statement by rounding.

On each grid $(1/k)\Z$ the shortest complete list of these tail bounds is the list of facets of the X-ray polytope, and that list grows faster than any polynomial in $k$. The facets are the subject of a companion paper \cite{ChaudharyKhandelwalFacets}, and nothing proved here depends on it.

\Cref{sec:setup} fixes notation and states the fibered reduction, and \cref{sec:grid} develops the grid picture and the duality. \Cref{sec:tails} proves the tail-sum characterization and the step counts, \cref{sec:density} treats bounded densities, and \cref{sec:windows} looks at three-point laws. Longer proofs, the four-atom laws, and a description of the exact computations that back the numerical statements are collected in the supplementary material, cited here with the prefix S.
\section{Setup and the fibered reduction}\label{sec:setup}

\subsection{Notation}
Throughout, `law' means probability distribution on $\R$, $\frc(x)=x-\lfloor x\rfloor$, and $\delta_x$ is the point mass at $x$. For laws $F,G$ with finite means, $F\le_{\mathrm{cx}}G$ (convex order) means $\int f\,dF\le\int f\,dG$ for every convex $f$; applied to random variables it refers to their laws. The stop-loss function of $F$ is $x\mapsto\int(s-x)^+\,F(ds)$, written $\pi_X(x)=\mathbb{E}(X-x)^+$ for $X\sim F$; for laws with equal means, $F\le_{\mathrm{cx}}G$ if and only if the stop-loss function of $F$ lies below that of $G$.

For laws $F,G$ let $\cD(F,G)$ be the set of laws of $X+Y$ over couplings with $X\sim F$, $Y\sim G$, and $\DU=\cD(U[0,1],U[0,1])$. The set $\cD(F,G)$ is convex and weakly compact, because the couplings of $F$ and $G$ form a convex, weakly compact set and $(x,y)\mapsto x+y$ is continuous (see also \cite[Lemma~2]{MaoWangWang2019}). For probability mass functions $p,q$ on $\Z$ let $\Sigma(p,q)$ be the set of laws of $I+J$ over couplings of $p$ and $q$; for a set $A$ of laws on $\Z$, $A+1$ is the set of their shifts by one. Let $\mathcal{U}_m$ be the uniform law on $\{0,\dots,m-1\}$ and
\[
P(m)=\Sigma(\mathcal{U}_m,\mathcal{U}_m),
\]
which by the Birkhoff--von Neumann theorem is the convex hull of the laws of $i+\sigma(i)$ over permutations $\sigma$ of $\{0,\dots,m-1\}$; $P(1)=\{\delta_0\}$. For $\lambda,\mu\ge 0$, $\lambda A+\mu B$ is the Minkowski sum of scaled sets, and a term with weight zero is dropped.

Fix $a\in\R$ and put $c=\frc(a)$ and $\varphi(r)=\frc(a-r)$ for $r\in[0,1)$. Then $\varphi$ is an everywhere-defined Borel involution of $[0,1)$ ($\varphi(r)=c-r$ on $[0,c]$ and $\varphi(r)=1+c-r$ on $(c,1)$), it preserves Lebesgue measure, and $r+\varphi(r)=c+\ind{r>c}$ for every $r$. For a random variable $X\sim F$ write $X=I+R$ with $I=\lfloor X\rfloor$ and $R=\frc(X)$, let $\nu(F)$ be the law of $R$, and let $r\mapsto p(r)$ be a regular conditional distribution of $I$ given $R=r$, that is, a probability kernel from $[0,1)$ to $\Z$ with $\mathbb{P}(I=i,\,R\in B)=\int_Bp(r)(\{i\})\,\nu(F)(dr)$ for all $i$ and Borel $B$. Define $\nu(G)$ and the kernel $q(\cdot)$ in the same way for $G$.

\Cref{tab:notation} collects the notation used throughout.

\begin{table}[ht]
\centering
\caption{Notation.}\label{tab:notation}
\small
\begin{tabular}{@{}lp{0.62\textwidth}l@{}}
\toprule
Symbol & Meaning & Defined in\\
\midrule
$\DU$ & laws of $X+Y$ over couplings of two standard uniforms & \cref{sec:setup}\\
$U$, $U_k$ & standard uniform; uniform on $\{0,\dots,k-1\}$ (law $\mathcal{U}_k$) & \cref{sec:setup}\\
$\le_{\mathrm{cx}}$, $\pi_X$ & convex order; stop-loss function $\pi_X(x)=\mathbb{E}(X-x)^+$ & \cref{sec:setup}\\
$P(m)$, $\Ph(m)$ & laws of $I+J$ for couplings of two copies of $\mathcal{U}_m$; $\Ph(m)=m\,P(m)$, the anti-diagonal X-rays of $m\times m$ doubly stochastic matrices & \cref{sec:setup,sec:grid}\\
$T$ & $kS-1$ for a law of $S$ on $\{1/k,\dots,(2k-1)/k\}$ & \cref{sec:grid}\\
$u\,\square\,u$ & infimal convolution of a potential $u$ on $[0,1]$ with itself & \cref{sec:duality}\\
$v$, $I(u)$ & grid infimal convolution of $u$ on $\{0,\dots,k-1\}$; the inequality $\sum_tv(t)x(t)\le2\sum_iu(i)$ & \cref{sec:gridineq}\\
$\eta_q$, $\xi_\ell$ & thresholds of a staircase and the thresholds of its tail-sum bound & \cref{sec:tails}\\
$S^\flat$, $\kappa$ & $S$ folded at $1$: $\kappa=\ind{S>1}$, $S^\flat=S-\kappa$ & \cref{cor:readings}\\
$\cG(b,a)$ & laws in $\DU$ carried by the lattice $a+b\Z$ & \cref{thm:B}\\
\bottomrule
\end{tabular}
\end{table}

\subsection{The reduction}

\begin{thm}[Fibered reduction]\label{thm:A}
Let $F,G$ be laws on $\R$, $a\in\R$, and $c=\frc(a)$.
\begin{enumerate}[label=(\roman*)]
\item If some coupling of $F$ and $G$ has $X+Y\in a+\Z$ almost surely, then $\nu(G)=\varphi_{\#}\nu(F)$.
\item If $\nu(G)=\varphi_{\#}\nu(F)$, the set of laws of $X+Y-c$ over couplings with $X+Y\in a+\Z$ almost surely equals the set of mixtures
\[
\int \mu^{+}(r)\,\nu(F)(dr),
\]
where $r\mapsto\mu(r)$ is Borel, $\mu(r)\in\Sigma\bigl(p(r),q(\varphi(r))\bigr)$ for $\nu(F)$-almost every $r$, and $\mu^{+}(r)$ is $\mu(r)$ shifted by $\ind{r>c}$. This set is convex.
\end{enumerate}
\end{thm}

The conclusion does not depend on the choice of the kernels. Two versions of $p$ differ only on a $\nu(F)$-null set $N$, and two versions of $q$ only on a $\nu(G)$-null set $N'$; the fibers in (ii) change only for $r\in N\cup\varphi^{-1}(N')$, and $\varphi^{-1}(N')$ is $\nu(F)$-null because $\varphi_\#\nu(F)=\nu(G)$. So the set of mixtures is the same.

The proof of (ii) for arbitrary marginals uses a measurable selection theorem; it is in Section~S1, which also shows that bounded marginals, the only case used here, need no selection theorem. Part (i) is immediate: on $\{X+Y\in a+\Z\}$ the fractional parts $R$ of $X$ and $R'$ of $Y$ satisfy $R+R'\in\{c,c+1\}$, so $R'=\frc(a-R)=\varphi(R)$.

\begin{lem}[Constant fibers collapse]\label{lem:M}
Let $E$ be a Borel set with $\nu(F)(E)=w>0$, and suppose the fiber $\Sigma(p(r),q(\varphi(r)))$ equals a fixed polytope $\Delta$ for $\nu(F)$-almost every $r\in E$. Then the set of integrals $\int_E\mu(r)\,\nu(F)(dr)$ over Borel selections $\mu$ is $w\Delta$.
\end{lem}

\begin{proof}
Constant selections give $w\Delta$. Conversely, $w^{-1}\int_E\mu(r)\,\nu(F)(dr)$ is the barycenter of a probability measure carried by the closed convex set $\Delta$, so it lies in $\Delta$.
\end{proof}

\subsection{Two uniforms}

\begin{thm}[Two uniforms, every spacing]\label{thm:B}
Let $X,Y\sim U[0,1]$, $b>0$, $L=1/b$, $n=\lfloor L\rfloor$, and $\theta=\frc(L)$. Let $\cG(b,a)$ be the set of laws of $S=X+Y$ over couplings with $S\in a+b\Z$ almost surely.
\begin{enumerate}[label=(\roman*)]
\item $\cG(b,a)\ne\emptyset$ if and only if $\theta=0$ or $1\in a+b\Z$.
\item If $\theta=0$, then with $c=\frc(a/b)$,
\[
\{\law(S/b-c):\law(S)\in\cG(b,a)\}=c\,P(n)+(1-c)\,(P(n)+1).
\]
\item If $\theta>0$ and $1\in a+b\Z$, then $\frc(a/b)=\theta$ and
\[
\{\law(S/b-\theta):\law(S)\in\cG(b,a)\}=\frac{\theta(n+1)}{L}\,P(n+1)+\frac{(1-\theta)n}{L}\,(P(n)+1).
\]
\end{enumerate}
\end{thm}

For $b=1$ ($L=n=1$, $\theta=0$), part (ii) gives the single law $c\,\delta_c+(1-c)\,\delta_{c+1}$ with $c=\frc(a)$, which is $\delta_1$ when $c=0$. For $b>1$ ($n=0$, $\theta=L$) the second term of (iii) has weight zero and is omitted; the first fiber is $P(1)=\{\delta_0\}$, so $\cG(b,a)=\{\delta_1\}$ when $1\in a+b\Z$ and is empty otherwise.

\begin{proof}
Scale by $1/b$ and apply \cref{thm:A} to $X/b$, $Y/b$ with $a'=a/b$ and $c=\frc(a')$. The fractional part $R$ of $X/b$ has density $N(r)/L$ on $[0,1)$ (almost everywhere), where $N(r)=\#\{i\ge 0:i+r<L\}=\lceil L-r\rceil$, which equals $n+1$ on $[0,\theta)$ and $n$ on $[\theta,1)$ ($n$ everywhere if $\theta=0$). Given $R=r$, $\lfloor X/b\rfloor$ is uniform on $\{0,\dots,N(r)-1\}$. $Y/b$ has the same decomposition.

(i), necessity. \Cref{thm:A}(i) requires $\varphi$ to preserve the density $N(\cdot)/L$, that is, $N(\varphi(r))=N(r)$ for almost every $r$. If $\theta>0$ this says that $\varphi$ maps $[0,\theta)$ onto itself up to a null set. Now $\varphi([0,\theta))$ is the arc $(c-\theta,c]$ mod 1, which coincides with $[0,\theta)$ up to a null set if and only if $c=\theta$, that is, if and only if $a/b-L\in\Z$, which is the same as $1\in a+b\Z$.

(i), sufficiency. If $\theta>0$ and $1\in a+b\Z$, the coupling $Y=1-X$ works. If $\theta=0$, the reversal $i\mapsto n-1-i$ on every fiber gives a law on the grid.

(ii) With $\theta=0$ every fiber is $P(n)$, so $\mu^+(r)$ ranges over $P(n)$ for $r\in[0,c]$ and over $P(n)+1$ for $r\in(c,1)$. \Cref{lem:M} on the two pieces gives the claim.

(iii) With $c=\theta>0$: on $[0,\theta)$ we have $N(r)=n+1$ and $\varphi(r)=\theta-r\in(0,\theta]$, so $N(\varphi(r))=n+1$ except at $r=0$; there is no carry, and the $\nu$-mass is $\theta(n+1)/L$. On $(\theta,1)$ we have $N(r)=N(\varphi(r))=n$ and a carry of one; the $\nu$-mass is $(1-\theta)n/L$. \Cref{lem:M} gives the stated sum, and the masses add to $(n+\theta)/L=1$.
\end{proof}

\begin{rem}\label{rem:B}
The necessity half of (i), and the fractional-part density argument behind it, are due to Mao et al.\ \cite[Lemma~4(i)]{MaoWangWang2019}; their constructions in Lemma~4(ii),(iii) are fiberwise selections in the sense of \cref{thm:A}. Their characterization of two-point laws in $\DU$ (a law on $\{a,a+b\}$ with mean one belongs to $\DU$ if and only if $1/b\in\N$) follows from (i) and (ii): a mean-one law on $\{a,a+b\}$ has $1$ strictly between its atoms, so $1\notin a+b\Z$ and (i) forces $\theta=0$; conversely, with $\lfloor a/b\rfloor=n-1$, the law $c\,\delta_{n-1}+(1-c)\,\delta_{n}$ in (ii), the image of the reversal on every fiber, is a law of $S/b-c$ on $\{a/b-c,\,a/b-c+1\}$ with mean $1/b-c$, hence it is the unique mean-one law on $\{a,a+b\}$ after rescaling. If $L=m/q$ is rational, then $a+b\Z\subseteq a+(b/q)\Z$ and $1/(b/q)=m$ is an integer, so (iii) is the slice of (ii) on the finer grid by the coordinates off the coarse grid. Part (iii) is needed for irrational $L$, and for rational $L$ it uses two polytopes, of sizes $n+1$ and $n$, instead of one of size $m$.
\end{rem}

\section{The grid \texorpdfstring{$(1/k)\Z$}{(1/k)Z} and potential duality}\label{sec:grid}

\subsection{Count coordinates}
It is convenient to work with $\Ph(m)=m\,P(m)$, the set of \emph{anti-diagonal X-rays} $x=(x(0),\dots,x(2m-2))$ of $m\times m$ doubly stochastic matrices $Z$, where $x(t)=\sum_{i+j=t}Z_{ij}$. It is a polytope of dimension $2m-3$ for $m\ge2$ (\cref{lem:affhull}). Inequalities on $\Ph(m)$ are stated in these count coordinates; dividing both sides by $m$ gives their probabilistic form.

\begin{prop}[Lattice laws and X-rays]\label{prop:C}
A law $F$ on $\{1/k,2/k,\dots,(2k-1)/k\}$ belongs to $\DU$ if and only if the vector with entries $x(t)=k\,F(\{(t+1)/k\})$, $t=0,\dots,2k-2$, lies in $\Ph(k)$. No law in $\DU$ supported on $(1/k)\Z$ charges $0$ or $2$.
\end{prop}

\begin{proof}
If $S=X+Y\in(1/k)\Z$ almost surely, put $I=\lfloor kX\rfloor$ and $J=\lfloor kY\rfloor$. The fractional parts of $kX$ and $kY$ then add to $1$ almost surely, so $kS=I+J+1$, and $(I,J)$ is a coupling of two uniform laws on $\{0,\dots,k-1\}$. Conversely, given such a coupling and an independent $U\sim U[0,1]$, the variables $X=(I+U)/k$ and $Y=(J+1-U)/k$ are standard uniform and $k(X+Y)=I+J+1$. Multiplying the coupling probabilities by $k$ gives a doubly stochastic matrix, whose X-ray is $x$. A mass at $0$ or $2$ would need an atom of a uniform marginal.
\end{proof}

This is \cref{thm:B}(ii) with $b=1/k$ and $a\in(1/k)\Z$, stated so that the polytope appears explicitly. It translates every linear inequality on $\Ph(k)$ into a statement about lattice laws in $\DU$, and conversely. Doubly stochastic matrices are the discrete copulas on the uniform grid \cite{PerroneSolusUhler2019}, which gives another reading of \cref{prop:C}.

The lattice laws are dense in $\DU$, and at a sharp rate: in the $\infty$-Wasserstein metric the Hausdorff distance between $\DU$ and its laws on $(1/k)\Z$ is exactly $1/k$ (Proposition~S2.1).

\subsection{Potential duality}\label{sec:duality}
We now introduce the inequalities used in the rest of the paper, first for $\DU$ itself. For a bounded function $u$ on $[0,1]$ let
\[
(u\,\square\,u)(s)=\inf\bigl\{u(x)+u(y):\ x,y\in[0,1],\ x+y=s\bigr\},\qquad s\in[0,2],
\]
the infimal convolution of $u$ with itself. If $u$ is continuous, so is $u\,\square\,u$. For every coupling, $(u\,\square\,u)(X+Y)\le u(X)+u(Y)$ pointwise, so
\begin{equation}\label{eq:potential}
\int(u\,\square\,u)\,dF\ \le\ 2\int_0^1u(x)\,dx\qquad\text{for every }F\in\DU,
\end{equation}
whenever $u\,\square\,u$ is Borel, for instance for continuous $u$ and for step functions with finitely many steps.

\Cref{thm:dual} below is Kantorovich duality for costs of the form $h(x+y)$ \cite{Kellerer1984,Villani2003}. Because the two marginals are equal, the two dual functions can be replaced by their average, so a single function $u$, which we call a \emph{potential}, indexes the constraints. The same symmetrization reduces a related dual problem for two uniforms to one potential in \cite{AnsariRockel2026}, and with indicator functions in place of $h$ the same duality gives the dual bounds for the distribution function of a sum \cite{EmbrechtsPuccetti2006}.

\begin{prop}[Kantorovich duality with one potential]\label{thm:dual}
A law $F$ belongs to $\DU$ if and only if $F([0,2])=1$ and \eqref{eq:potential} holds for every continuous $u$ on $[0,1]$.
\end{prop}

\begin{proof}
Necessity is the display above. For sufficiency, let $F$ be a law on $[0,2]$ satisfying \eqref{eq:potential} and suppose $F\notin\DU$. The set $\DU$ is convex and weakly closed (\cref{sec:setup}), and the laws on $[0,2]$ carry the weak topology induced by $C[0,2]$; by the Hahn--Banach theorem there is $h\in C[0,2]$ with $\int h\,dF>V(h)$, where $V(h)=\sup_{G\in\DU}\int h\,dG=\sup\mathbb{E}h(X+Y)$ over couplings of two standard uniforms. By Kantorovich duality for the continuous cost $-h(x+y)$ on $[0,1]^2$ \cite[Chapter~1]{Villani2003} (see \cite{Kellerer1984} for duality in general marginal problems), for every $n\ge1$ there are $\alpha,\beta\in C[0,1]$ with $\alpha(x)+\beta(y)\ge h(x+y)$ for all $x,y$ and $\int\alpha+\int\beta\le V(h)+1/n$. Put $u=(\alpha+\beta)/2$. Adding the constraint at $(x,y)$ and at $(y,x)$ gives $u(x)+u(y)\ge h(x+y)$, so $h\le u\,\square\,u$ on $[0,2]$, and $2\int u=\int\alpha+\int\beta$. Hence $\int h\,dF\le\int(u\,\square\,u)\,dF\le2\int u\le V(h)+1/n$ for every $n$, a contradiction.
\end{proof}

The content of \cref{thm:dual} lies in deciding which potentials are needed. Three classes of potentials have a direct probabilistic meaning.

\begin{cor}[Three readings]\label{cor:readings}
Let $F$ be a law on $[0,2]$ and $S\sim F$, put $\kappa=\ind{S>1}$ and $S^\flat=S-\kappa$, and let $U$ be standard uniform.
\begin{enumerate}[label=(\roman*)]
\item \emph{Convex potentials.} If $u$ is convex, $(u\,\square\,u)(s)=2u(s/2)$, and \eqref{eq:potential} holds for every continuous convex $u$ if and only if $S\le_{\mathrm{cx}}2U$.
\item \emph{Concave potentials.} If $u$ is concave, $(u\,\square\,u)(s)=u(s-\ind{s>1})+u(\ind{s>1})$, and \eqref{eq:potential} holds for every continuous concave $u$ if and only if $\law(U)\le_{\mathrm{cx}}\frac12\bigl(\law(S^\flat)+\law(\kappa)\bigr)$.
\item \emph{One step.} If $F\in\DU$ and $0<\eta<1$, then
\begin{equation}\label{eq:onestep}
\mathbb{P}(S\ge2\eta)+\mathbb{P}(S\ge1+\eta)\ \le\ 2(1-\eta),
\end{equation}
which is \eqref{eq:potential} for $u=\ind{x\ge\eta}$. Equality holds for the coupling $Y=X$ on $\{X<\eta\}$ and $Y=1+\eta-X$ on $\{X\ge\eta\}$.
\end{enumerate}
\end{cor}

\begin{proof}
(i) By convexity $u(x)+u(s-x)\ge2u(s/2)$, and $x=s/2$ is admissible. So \eqref{eq:potential} for convex $u$ reads $\mathbb{E}u(S/2)\le\mathbb{E}u(U)$, and for all continuous convex $u$ on $[0,1]$ this is $S/2\le_{\mathrm{cx}}U$. (ii) The function $x\mapsto u(x)+u(s-x)$ is concave on the admissible interval $[\max(0,s-1),\min(1,s)]$, so it is minimal at an endpoint; the two endpoints give the same value, $u(0)+u(s)$ if $s\le1$ and $u(s-1)+u(1)$ if $s\ge1$. So the inequalities read $\mathbb{E}w(S^\flat)+\mathbb{E}w(\kappa)\ge2\,\mathbb{E}w(U)$ for all continuous convex $w=-u$, which is the stated convex order; the affine potentials, which are both convex and concave, force equal means, $\mathbb{E}S^\flat+\mathbb{E}\kappa=\mathbb{E}S=1$. (iii) On $[0,2]$, a split $s=x+y$ with $x,y<\eta$ exists if and only if $s<2\eta$, and a split with $x<\eta$ and $y\le1$ exists if and only if $s<1+\eta$, so $u\,\square\,u=\ind{s\ge2\eta}+\ind{s\ge1+\eta}$ and $2\int u=2(1-\eta)$. The coupling maps $[\eta,1]$ onto itself by a reflection, so $Y$ is uniform; it gives $S=2X<2\eta$ with probability $\eta$ and $S=1+\eta$ with probability $1-\eta$, so both probabilities in \eqref{eq:onestep} equal $1-\eta$.
\end{proof}

Part (i) is the classical necessary condition: every sum is dominated in convex order by the comonotonic sum, and for three or more uniforms this condition is also sufficient \cite{MaoWangWang2019}. Part (ii) is the analogous condition on the other side, obtained by folding $S$ at $1$. For two uniforms neither is sufficient, and part (iii) gives a constraint that neither implies (\cref{ex:twopoint}). Individually, $\sup_{\DU}\mathbb{P}(S\ge s)=\min(1,2-s)$ \cite{Makarov1981,Ruschendorf1982}, so the two worst cases in \eqref{eq:onestep} add up to $\min(1,2-2\eta)+1-\eta$, which exceeds $2(1-\eta)$ for every $\eta\in(0,1)$. The worst cases at the thresholds $2\eta$ and $1+\eta$ therefore cannot occur together, and \eqref{eq:onestep} is the sharp joint bound. Reflecting $S\mapsto2-S$ gives the lower-tail version $\mathbb{P}(S\le2\epsilon)+\mathbb{P}(S\le\epsilon)\le2\epsilon$ for $0<\epsilon<1$. For $p$-values $X,Y$ with arbitrary dependence and their average $\bar p=(X+Y)/2$ it reads $\mathbb{P}(\bar p\le\epsilon)+\mathbb{P}(\bar p\le\epsilon/2)\le2\epsilon$, a joint form of the factor $2$ for merging by the arithmetic mean \cite{VovkWang2020}. This holds for exactly uniform $p$-values, and hence also for valid ones ($\mathbb{P}(X\le t)\le t$), since after randomization a valid $p$-value dominates a uniform variable almost surely. Which potentials are needed in \cref{thm:dual}, beyond these three classes, is the subject of \cref{sec:tails,sec:density}.

\begin{ex}\label{ex:twopoint}
The law $F=\frac12\delta_{0.6}+\frac12\delta_{1.4}$ has mean one and satisfies both (i) and (ii): its stop-loss function $\mathbb{E}(S-x)^+$ stays below $(2-x)^2/4=\mathbb{E}(2U-x)^+$ (the gap is $x^2/4$ on $[0,0.6]$ and at least $0.05$ on $[0.6,1.4]$), and the folded law $\frac14(\delta_0+\delta_{0.4}+\delta_{0.6}+\delta_1)$ has a stop-loss function above that of $U$ at each of its atoms, hence everywhere. But \eqref{eq:onestep} with $\eta=0.3$ gives $1+\frac12>1.4$. Hence $F\notin\DU$, in agreement with the two-point theorem of \cite{MaoWangWang2019} ($1/b=1.25\notin\N$). On the grid $k=10$ the inequality used is $I(u)$ for $u=\ind{i\ge3}$ (\cref{sec:gridineq}), which defines a facet of $\Ph(10)$: the companion paper \cite{ChaudharyKhandelwalFacets} shows that every single step $u=\ind{i\ge a}$, $1\le a\le k-1$, defines a facet of $\Ph(k)$ for every $k\ge3$.
\end{ex}

\subsection{Grid inequalities}\label{sec:gridineq}
On the grid, step functions give a finite description. For $u:\{0,\dots,k-1\}\to\R$ put $v(t)=\min_{i+j=t}\bigl(u(i)+u(j)\bigr)$, $t=0,\dots,2k-2$, the discrete infimal convolution of $u$ with itself, and
\[
I(u):\qquad \sum_{t=0}^{2k-2}v(t)\,x(t)\ \le\ 2\sum_{i=0}^{k-1}u(i).
\]
Every permutation $\sigma$ satisfies $v(i+\sigma(i))\le u(i)+u(\sigma(i))$ termwise, so $I(u)$ is valid on $\Ph(k)$. Let $\tilde u$ be the upper semicontinuous step function on $[0,1]$ equal to $u(i)$ on $(i/k,(i+1)/k)$ and to $\max(u(i-1),u(i))$ at $i/k$, where $u(-1)=u(k)=-\infty$. A split $(t+1)/k=x+y$ with $x$ in the open cell $(i/k,(i+1)/k)$ has $y$ in the open cell of $t-i$, and a split at two lattice points costs at least as much as one of the neighboring open splits, so $(\tilde u\,\square\,\tilde u)((t+1)/k)=v(t)$; also $2\int\tilde u=\frac2k\sum_iu(i)$. So for a law of $S$ on $\{1/k,\dots,(2k-1)/k\}$, with $T=kS-1$ and $x(t)=k\,\mathbb{P}(T=t)$, \eqref{eq:potential} for $\tilde u$ reads $\mathbb{E}\,v(T)\le2\,\mathbb{E}\,u(U_k)$, which is $I(u)$ divided by $k$.

\begin{lem}[Affine hull]\label{lem:affhull}
For $k\ge2$ the affine hull of $\Ph(k)$ is $\{x:\sum_tx(t)=k,\ \sum_tt\,x(t)=k(k-1)\}$, so $\dim\Ph(k)=2k-3$. If $u'(i)=u(i)+\alpha+\beta i$, then $v'(t)=v(t)+2\alpha+\beta t$, and $I(u')$ and $I(u)$ agree on the affine hull.
\end{lem}

\begin{proof}
Both equations hold for the X-ray of every permutation, since $\sum_i(i+\sigma(i))=k(k-1)$, hence on $\Ph(k)$. Their solution set has dimension $2k-3$, and its direction space is spanned by the $2k-3$ independent second differences $e_{t-1}-2e_t+e_{t+1}$, $1\le t\le2k-3$. Each of them is a difference of two X-rays of permutations: write $t=i+s$ with $0\le i\le k-2$ and $1\le s\le k-1$, take a permutation $\sigma$ with $\sigma(i)=s$ and $\sigma(i+1)=s-1$, and compose it with the transposition of $i$ and $i+1$; the sums $t,t$ at rows $i,i+1$ become $t-1,t+1$. For the second claim, the right side changes by $2k\alpha+\beta k(k-1)$, which equals the change of the left side on the affine hull.
\end{proof}

\begin{lem}[Grid duality]\label{lem:gridduality}
Let $x\ge0$ on $\{0,\dots,2k-2\}$. Then $x\in\Ph(k)$ if and only if $I(u)$ holds at $x$ for every $u:\{0,\dots,k-1\}\to\R$.
\end{lem}

\begin{proof}
Necessity: if $x$ is the X-ray of a doubly stochastic $Z$, then $\sum_tv(t)x(t)=\sum_{i,j}Z_{ij}\,v(i+j)\le\sum_{i,j}Z_{ij}\bigl(u(i)+u(j)\bigr)=2\sum_iu(i)$. Sufficiency: $\Ph(k)$ is a polytope, so if $x\notin\Ph(k)$ there is $c\in\R^{2k-1}$ with $\sum_tc(t)x(t)>\max\{\sum_{i,j}c(i+j)Z_{ij}\}$, the maximum over doubly stochastic $Z$. By linear programming duality for the assignment problem this maximum equals $\min\{\sum_i\alpha(i)+\sum_j\beta(j):\ \alpha(i)+\beta(j)\ge c(i+j)\ \forall i,j\}$. For an optimal pair put $u=(\alpha+\beta)/2$; adding the constraints at $(i,j)$ and $(j,i)$ gives $u(i)+u(j)\ge c(i+j)$, so $c\le v$, and $2\sum_iu(i)=\sum\alpha+\sum\beta$. Since $x\ge0$, $\sum_tc(t)x(t)\le\sum_tv(t)x(t)\le2\sum_iu(i)$, a contradiction.
\end{proof}

\Cref{lem:gridduality} is the grid form of \cref{thm:dual}; Li et al.\ \cite[Proposition~3.1]{LiLiuYao2026} prove it in the convention of diagonal X-rays. They also show \cite[Proposition~3.4]{LiLiuYao2026} that finitely many potentials suffice: the rays of a tropical fan. On each grid, then, the lattice part of $\DU$ is cut out by nonnegativity, the affine hull, and finitely many inequalities $I(u)$, and the minimal list of them is the list of facets of $\Ph(k)$ other than the coordinate facets $x(t)\ge0$. The facets are the subject of the companion paper \cite{ChaudharyKhandelwalFacets}; no result of this paper depends on it. Here we need only the potentials that give convex order.

For $1\le j\le k-2$ let $h_j(i)=(i-j)^+$, the \emph{hinge} at $j$. Write $U_k$ for a random variable with law $\mathcal{U}_k$.

\begin{prop}[Hinges]\label{prop:hinges}
Let $k\ge3$ and let $x\ge0$ lie on the affine hull of $\Ph(k)$, viewed as the law of $T$ with $\mathbb{P}(T=t)=x(t)/k$.
\begin{enumerate}[label=(\roman*)]
\item For $u=h_j$, $v(t)=(t-2j)^+$, and $I(h_j)$ reads $\mathbb{E}(T-2j)^+\le\mathbb{E}(2U_k-2j)^+$. The $k-2$ inequalities $I(h_j)$ hold if and only if $T\le_{\mathrm{cx}}2U_k$.
\item For $u=-h_j$, $I(-h_j)$ reads $\mathbb{E}(T^\flat-j)^++\mathbb{E}(\kappa_k-j)^+\ge2\,\mathbb{E}(U_k-j)^+$, where $\kappa_k=(k-1)\ind{T\ge k}$ and $T^\flat=T-\kappa_k$ (grid analogues of $\kappa$ and $S^\flat$, not their rescalings). The $k-2$ inequalities $I(-h_j)$ hold if and only if $U_k\le_{\mathrm{cx}}\frac12\bigl(\law(T^\flat)+\law(\kappa_k)\bigr)$.
\end{enumerate}
\end{prop}

\begin{proof}
(i) As in \cref{cor:readings}(i), the minimizing split of $t$ is the balanced one, which gives $v(t)=(t-2j)^+$ and $2\sum_iu(i)=k\,\mathbb{E}(2U_k-2j)^+$. For integer-valued laws with the same mean, $T\le_{\mathrm{cx}}2U_k$ means $\mathbb{E}(T-m)^+\le\mathbb{E}(2U_k-m)^+$ for every integer $m$. The even values $m=2j$ with $1\le j\le k-2$ are the inequalities $I(h_j)$; $m\le0$ and $m\ge2k-2$ hold trivially on the affine hull; and an odd $m$ follows from its neighbors, because $\mathbb{E}(T-m)^+$ is convex in $m$ while $\mathbb{E}(2U_k-m)^+$ is linear between consecutive even integers. (ii) As in \cref{cor:readings}(ii), the minimizing split of $t$ is $(0,t)$ for $t\le k-1$ and $(t-k+1,k-1)$ for $t\ge k-1$, so $v(T)=u(T^\flat)+u(\kappa_k)$, and the inequality follows by expanding $u=-h_j$. The means agree because $\mathbb{E}T^\flat+\mathbb{E}\kappa_k=\mathbb{E}T=k-1=2\,\mathbb{E}U_k$, and for integer-valued laws convex order is decided by the hinges at the integers $1,\dots,k-2$.
\end{proof}

The companion paper \cite{ChaudharyKhandelwalFacets} shows that the $2(k-2)$ inequalities $I(h_j)$ and $I(-h_j)$ are facets of $\Ph(k)$ and that they are the only facets given by convex or concave potentials. For lattice laws with small atoms, convex order alone decides membership. For $u:\{0,\dots,k-1\}\to\R$ let $\varphi$ be its \emph{convex minorant}: the largest convex function on $[0,k-1]$ with $\varphi(i)\le u(i)$ for all $i$. It is the lower boundary of the convex hull of the points $(i,u(i))$, so it is piecewise linear, its breakpoints are integers, and $\varphi=u$ at $0$, at $k-1$, and at every breakpoint. Call these the \emph{contact points} $0=a_0<a_1<\dots<a_m=k-1$; then $\varphi$ is affine on each $[a_{l-1},a_l]$ and $\varphi(a_l)=u(a_l)$.

\begin{lem}[Convex minorant]\label{lem:minorant}
Let $u:\{0,\dots,k-1\}\to\R$, let $\varphi$ be its convex minorant, and let $w=u-\varphi\ge0$ on $\{0,\dots,k-1\}$. Then
\[
\sum_{t=0}^{2k-2}\bigl(v(t)-2\varphi(t/2)\bigr)^+\ \le\ 2\sum_{i=0}^{k-1}w(i).
\]
\end{lem}

\begin{proof}
Fix a piece $[a,b]=[a_{l-1},a_l]$, on which $\varphi$ is affine, and recall $w(a)=w(b)=0$. If $2a\le t\le a+b$, then $a$ and $t-a$ lie in $[a,b]$ and are symmetric about $t/2$, so $\varphi(a)+\varphi(t-a)=2\varphi(t/2)$ and
\[
v(t)\le u(a)+u(t-a)=2\varphi(t/2)+w(a)+w(t-a)=2\varphi(t/2)+w(t-a).
\]
Likewise $v(t)\le2\varphi(t/2)+w(t-b)$ for $a+b\le t\le2b$. As $w\ge0$, summing over $t\in[2a,2b]$, with $t=a+b$ counted in both ranges, gives $\sum_{t=2a}^{2b}(v(t)-2\varphi(t/2))^+\le2\sum_{i=a}^{b}w(i)$. The intervals $[2a_{l-1},2a_l]$ cover $\{0,\dots,2k-2\}$; at a shared endpoint $t=2a_l$ the summand is $(v(2a_l)-2u(a_l))^+=0$, and $w$ vanishes at the contact points shared by adjacent pieces. Summing over $l$ gives the claim.
\end{proof}

\begin{cor}[Small atoms]\label{cor:smallatoms}
A mean-one law on $\{1/k,\dots,(2k-1)/k\}$ whose atoms are all at most $1/k$ belongs to $\DU$ if and only if $kS-1\le_{\mathrm{cx}}2U_k$. Equivalently, a vector $x$ with $0\le x(t)\le1$, $\sum_tx(t)=k$, and $\sum_tt\,x(t)=k(k-1)$ lies in $\Ph(k)$ if and only if the $k-2$ hinge inequalities $I(h_j)$ hold.
\end{cor}

\begin{proof}
By \cref{prop:C} the two statements are the same, with $T=kS-1$ and $x(t)=k\,\mathbb{P}(T=t)$, and by \cref{prop:hinges}(i) the hinge inequalities say $T\le_{\mathrm{cx}}2U_k$. \emph{Necessity.} If $T=I+J$ with $I,J\sim U_k$ and $f$ is convex, then $f(T)\le\frac12\bigl(f(2I)+f(2J)\bigr)$, so $\mathbb{E}f(T)\le\mathbb{E}f(2U_k)$. \emph{Sufficiency.} Let $T\le_{\mathrm{cx}}2U_k$ with $\mathbb{P}(T=t)\le1/k$ for every $t$. By \cref{lem:gridduality} it suffices to show $\mathbb{E}\,v(T)\le2\,\mathbb{E}\,u(U_k)$ for every $u$. Let $\varphi$ be the convex minorant of $u$ and $w=u-\varphi$. Then
\begin{align*}
\mathbb{E}\,v(T)
&\le 2\,\mathbb{E}\,\varphi(T/2)+\mathbb{E}\bigl(v(T)-2\varphi(T/2)\bigr)^+\\
&\le 2\,\mathbb{E}\,\varphi(U_k)+\frac1k\sum_{t=0}^{2k-2}\bigl(v(t)-2\varphi(t/2)\bigr)^+\\
&\le 2\,\mathbb{E}\,\varphi(U_k)+\frac2k\sum_{i=0}^{k-1}w(i)\ =\ 2\,\mathbb{E}\,u(U_k).
\end{align*}
The first line is $y\le y^+$. In the second, $s\mapsto\varphi(s/2)$ is convex, so $\mathbb{E}\varphi(T/2)\le\mathbb{E}\varphi(2U_k/2)$ by $T\le_{\mathrm{cx}}2U_k$, and $\mathbb{P}(T=t)\le1/k$ bounds the expectation of the nonnegative variable. The third line is \cref{lem:minorant}.
\end{proof}

\begin{rem}
The bound on the atoms enters only in the second line. Without it the same chain gives, for every $x\ge0$ on the affine hull with $T\le_{\mathrm{cx}}2U_k$ and every $u$,
\[
\sum_tv(t)x(t)-2\sum_iu(i)\ \le\ \sum_t\bigl(x(t)-1\bigr)^+\bigl(v(t)-2\varphi(t/2)\bigr)^+,
\]
so a violated inequality $I(u)$ must put its excess on anti-diagonals carrying more than one unit. Remark~S4.1 shows that the bound $1$ cannot be raised to $101/100$. For an equivalent statement through two families of linear inequalities, proved by a different route, see also \cite[Corollary~4.4]{LiLiuYao2026}. \Cref{cor:smallatoms} and \cref{lem:gridduality} are proved here without that reference, and \cref{thm:density} rests only on them.
\end{rem}
\section{Joint tail bounds}\label{sec:tails}

By \cref{thm:dual}, membership in $\DU$ is decided by the potential inequalities \eqref{eq:potential}. This section and the next ask which potentials are needed. Staircases, read as joint bounds on sums of tail probabilities, suffice for every law (\cref{thm:tails}); on a fixed grid they need repeated thresholds. The number of steps cannot be bounded in terms of the number of atoms: two- and three-point laws with small spacing need many steps (\cref{thm:twopoint,thm:threepoint}). \Cref{sec:density} treats laws with a bounded density. Some results and proofs here cite \cref{sec:windows}, which uses only \cref{thm:B} and \cref{cor:readings}(iii), so the order of the argument is not circular.

\subsection{Tail sums characterize \texorpdfstring{$\DU$}{the set}}
For thresholds $0<\eta_1\le\dots\le\eta_r\le1$, repetitions allowed, put $\eta_{r+1}=1$, let $u(x)=\#\{q:\eta_q\le x\}$ on $[0,1]$, and
\[
\xi_\ell=\max\bigl\{\eta_q+\eta_{q'}:\ q+q'=\ell+1,\ 1\le q,q'\le r+1\bigr\},\qquad 1\le\ell\le2r.
\]
We call $u$ the \emph{staircase} with these thresholds and $r$ its number of (unit) steps. The thresholds are nondecreasing in $\ell$: a pair attaining $\xi_\ell$ has an index $q\le r$, and raising it by one does not decrease the sum.

\begin{lem}[Staircases with repeated thresholds]\label{lem:tailsgen}
For $s\in[0,2]$, $(u\,\square\,u)(s)=\#\{\ell\le 2r:\ \xi_\ell\le s\}$, and $\int_0^1u=\sum_{q=1}^r(1-\eta_q)$. Consequently, for every law $F\in\DU$ and $S\sim F$,
\begin{equation}\label{eq:tailsgen}
\sum_{\ell=1}^{2r}\mathbb{P}(S\ge\xi_\ell)\ \le\ 2\sum_{q=1}^{r}(1-\eta_q).
\end{equation}
\end{lem}

\begin{proof}
For $0\le i\le r$ let $A_i=\{x\in[0,1]:u(x)\le i\}$. Since the thresholds are sorted, $u(x)\ge i+1$ if and only if $\eta_{i+1}\le x$, so $A_i=[0,\eta_{i+1})$ for $i<r$, and $A_r=[0,1]$. Let $1\le\ell\le2r$. We claim that $(u\,\square\,u)(s)\le\ell-1$ if and only if $s\in A_i+A_j$ for some $i,j\ge0$ with $i+j=\ell-1$ and $i,j\le r$. If $s=x+y$ with $u(x)+u(y)\le\ell-1$, put $i=u(x)\le r$ and $j=\ell-1-i\ge u(y)$; if $j>r$, replace $(i,j)$ by $(\ell-1-r,r)$, which keeps $x\in A_{\ell-1-r}\supseteq A_i$ and $y\in A_r$. The converse is clear. Because every threshold is positive, $A_i+A_j=[0,\eta_{i+1}+\eta_{j+1})$ for $i,j<r$ and $A_r+A_j=[0,1+\eta_{j+1})$ for $j<r$; the case $i=j=r$ does not occur because $\ell\le 2r$. So $(u\,\square\,u)(s)\le\ell-1$ if and only if $s<\xi_\ell$. Since $u\le r$, $u\,\square\,u\le 2r$ on $[0,2]$, which gives the formula, and $u\,\square\,u$ is a right-continuous step function, hence Borel. The integral is immediate, and \eqref{eq:tailsgen} is \eqref{eq:potential} for $u$.
\end{proof}

A threshold equal to $1$ adds nothing to the right side of \eqref{eq:tailsgen}. The single step ($r=1$) is \eqref{eq:onestep}, with $\xi_1=2\eta_1$ and $\xi_2=1+\eta_1$. Read this way, a staircase inequality bounds the expected number of thresholds $\xi_1\le\dots\le\xi_{2r}$ that $S$ reaches.

\begin{thm}[Tail sums characterize $\DU$]\label{thm:tails}
Let $F$ be a law on $[0,2]$ and $S\sim F$. The following are equivalent.
\begin{enumerate}[label=(\roman*)]
\item $F\in\DU$.
\item $\mathbb{E}S=1$, and \eqref{eq:tailsgen} holds for every $r\ge1$ and all thresholds $0<\eta_1\le\dots\le\eta_r\le1$.
\item $\mathbb{E}S=1$, and \eqref{eq:tailsgen} holds for every $r\ge1$ and all thresholds $0<\eta_1<\dots<\eta_r<1$.
\end{enumerate}
\end{thm}

\begin{proof}
(i)$\Rightarrow$(ii) is \cref{lem:tailsgen}, and every law in $\DU$ has mean one. (ii)$\Rightarrow$(iii) is trivial.

(iii)$\Rightarrow$(ii). Let $0<\eta_1\le\dots\le\eta_r\le1$, put $c=\eta_1/(r+1)$ and $\eta^{(n)}_q=\eta_q-(r+1-q)c/n$ for $n\ge1$. Then $\eta^{(n)}_1\ge\eta_1-rc>0$, $\eta^{(n)}_{q+1}-\eta^{(n)}_q\ge c/n>0$ and $\eta^{(n)}_r\le 1-c/n<1$, so (iii) applies to $\eta^{(n)}$. Each $\eta^{(n)}_q$ increases to $\eta_q$ as $n\to\infty$, and $\xi_\ell$ is a maximum of sums of thresholds, so the corresponding thresholds $\xi^{(n)}_\ell$ increase to $\xi_\ell$. The events $\{S\ge\xi^{(n)}_\ell\}$ decrease, with intersection $\{S\ge\xi_\ell\}$, so $\mathbb{P}(S\ge\xi_\ell^{(n)})\to\mathbb{P}(S\ge\xi_\ell)$, while the right side of \eqref{eq:tailsgen} converges. So \eqref{eq:tailsgen} holds for $\eta$.

(ii)$\Rightarrow$(i). By \cref{thm:dual} it suffices to prove \eqref{eq:potential} for every continuous $u$ on $[0,1]$. We reduce to staircases in three steps.

\emph{Affine changes.} If $w(x)=u(x)+\alpha+\beta x$, then $(w\,\square\,w)(s)=(u\,\square\,u)(s)+2\alpha+\beta s$, so the left side of \eqref{eq:potential} changes by $2\alpha+\beta\,\mathbb{E}S=2\alpha+\beta$, and so does the right side, $2\int w=2\int u+2\alpha+\beta$. So \eqref{eq:potential} holds for $u$ if and only if it holds for $w$.

\emph{Uniform limits.} If $\|u-w\|_\infty\le\varepsilon$, then $u(x)+u(y)\le w(x)+w(y)+2\varepsilon$ for all $x,y$, so $|(u\,\square\,u)-(w\,\square\,w)|\le2\varepsilon$ on $[0,2]$, and $|2\int u-2\int w|\le2\varepsilon$. So if \eqref{eq:potential} holds for $w$, its two sides for $u$ differ from those for $w$ by at most $2\varepsilon$ each, and \eqref{eq:potential} holds for $u$ up to $4\varepsilon$. By the Weierstrass theorem a continuous $u$ is a uniform limit of polynomials, which are Lipschitz, so it suffices to treat Lipschitz $u$. If $u$ is Lipschitz with constant $\beta$, then $u(x)+\beta x-u(0)$ is nondecreasing, so by the affine step it suffices to treat continuous nondecreasing $w$ with $w(0)=0$.

\emph{Staircases.} For such $w$ and an integer $N\ge1$ let $R=\lfloor Nw(1)\rfloor$ and, for $1\le m\le R$, $\eta_m=\min\{x\in[0,1]:w(x)\ge m/N\}$. By continuity $w(\eta_m)=m/N$, and $\eta_m>0$ because $w(0)=0$; hence $0<\eta_1<\dots<\eta_R\le1$. Since $w$ is nondecreasing, $\{x:w(x)\ge m/N\}=[\eta_m,1]$, so
\[
\lfloor Nw(x)\rfloor=\#\{m\ge1:\ w(x)\ge m/N\}=\#\{m\le R:\ \eta_m\le x\}
\]
is the staircase with thresholds $\eta_1,\dots,\eta_R$ (zero if $R=0$). By (ii) and \cref{lem:tailsgen}, \eqref{eq:potential} holds for it, hence for $w_N=\lfloor Nw\rfloor/N$, since $(cv)\,\square\,(cv)=c\,(v\,\square\,v)$ for $c>0$. As $\|w-w_N\|_\infty\le1/N$, \eqref{eq:potential} holds for $w$ up to $4/N$ for every $N$, hence exactly.
\end{proof}

For three or more uniforms convex order decides membership \cite{MaoWangWang2019}; \cref{thm:tails} says that for two uniforms the joint tail bounds \eqref{eq:tailsgen} take its place. The convex-order condition is implied by them, and \cref{ex:twopoint} shows that it is strictly weaker. In the continuum distinct thresholds suffice, by (iii). The staircase form is also the practical one: to show that a given law is not attainable, it is enough to exhibit a few thresholds and compare two numbers, which is all that \cref{ex:worked} below does.

\begin{rem}[Not every bound is sharp]\label{rem:notsharp}
Not every inequality \eqref{eq:tailsgen} is attained on $\DU$. The thresholds $(\frac15,\frac25,\frac12,\frac45)$ and $(\frac15,\frac25,\frac35,\frac45)$ give the same $\xi_1,\dots,\xi_8$, namely $(\frac25,\frac35,\frac45,1,\frac65,\frac75,\frac85,\frac95)$, with right sides $\frac{21}5$ and $4$, so the first bound is never attained. Single-step bounds are attained for every $\eta$ (\cref{cor:readings}(iii)), and on a grid every bound that defines a facet of $\Ph(k)$ is attained, because a facet is a nonempty face.
\end{rem}

The following example shows the theorem at work on a law with a density, for which no lattice argument applies.

\begin{ex}\label{ex:worked}
Let $S$ have density $4$ on the four intervals $[\frac7{20},\frac{33}{80}]$, $[\frac{61}{80},\frac{33}{40}]$, $[\frac{47}{40},\frac{99}{80}]$, and $[\frac{127}{80},\frac{33}{20}]$, each of length $\frac1{16}$ and mass $\frac14$. Then $\mathbb{E}S=1$ and $S\le_{\mathrm{cx}}2U$; this is the law of \cref{thm:equal} with $r=2$ and $c=4$. The folded condition of \cref{cor:readings}(ii) holds, and so does every single-step bound \eqref{eq:onestep}: for $0<\eta\le\frac{33}{40}$ the left side of \eqref{eq:onestep} falls short of the right side by at least $\frac3{40}$, with equality at $\eta=\frac{47}{80}$, and for $\eta>\frac{33}{40}$ the left side is $0$. None of the three readings of \cref{cor:readings} excludes $S$. The two-step staircase with thresholds $(\frac7{40},\frac{47}{80})$ has $(\xi_1,\dots,\xi_4)=(\frac7{20},\frac{61}{80},\frac{47}{40},\frac{127}{80})$, the left ends of the four intervals, so the left side of \eqref{eq:tailsgen} is $1+\frac34+\frac12+\frac14=\frac52$, while the right side is $2(\frac{33}{40}+\frac{33}{80})=\frac{99}{40}$. The violation is $\frac1{40}$, and $\law(S)\notin\DU$.
\end{ex}

\Cref{thm:tails} extends to two copies of any law on a compact interval. For a law $F$ write $\cD(F,F)$ as in \cref{sec:setup}.

\begin{prop}[General marginals]\label{prop:generalF}
Let $F$ be a law on $[a,b]$, $a<b$, $X\sim F$, and let $G$ be a law on $[2a,2b]$, $S\sim G$. For thresholds $a<\eta_1\le\dots\le\eta_r\le b$ put $\eta_{r+1}=b$ and $\xi_\ell=\max\{\eta_q+\eta_{q'}:q+q'=\ell+1,\ 1\le q,q'\le r+1\}$. Then $G\in\cD(F,F)$ if and only if $\mathbb{E}S=2\,\mathbb{E}X$ and
\[
\sum_{\ell=1}^{2r}\mathbb{P}(S\ge\xi_\ell)\ \le\ 2\sum_{q=1}^{r}\mathbb{P}(X\ge\eta_q)
\]
for every $r\ge1$ and all such thresholds; it suffices to take $a<\eta_1<\dots<\eta_r<b$.
\end{prop}

\begin{proof}
The proof of \cref{thm:tails} carries over, with $[0,1]$ replaced by $[a,b]$, $\int u$ by $\mathbb{E}u(X)$, and the mean condition by $\mathbb{E}S=2\,\mathbb{E}X$. \emph{Duality.} The proof of \cref{thm:dual} uses only compactness and the equality of the two marginals: separating $G$ from the convex weakly compact set $\cD(F,F)$ gives $h\in C[2a,2b]$, Kantorovich duality on $[a,b]^2$ gives continuous $\alpha,\beta$ with $\alpha(x)+\beta(y)\ge h(x+y)$, and $u=(\alpha+\beta)/2$ satisfies $h\le u\,\square\,u$ and $2\,\mathbb{E}u(X)=\mathbb{E}\alpha(X)+\mathbb{E}\beta(X)$, with $u\,\square\,u$ now an infimum over $[a,b]$. So $G\in\cD(F,F)$ if and only if $\mathbb{E}(u\,\square\,u)(S)\le2\,\mathbb{E}u(X)$ for every continuous $u$ on $[a,b]$. \emph{Staircases.} For $u(x)=\#\{q:\eta_q\le x\}$ on $[a,b]$, the proof of \cref{lem:tailsgen} with $A_i=[a,\eta_{i+1})$ for $i<r$ and $A_r=[a,b]$ gives $(u\,\square\,u)(s)=\#\{\ell\le2r:\xi_\ell\le s\}$, and $\mathbb{E}u(X)=\sum_q\mathbb{P}(X\ge\eta_q)$; this gives necessity. \emph{Sufficiency.} Adding $\alpha+\beta x$ to $u$ changes the left side by $2\alpha+\beta\,\mathbb{E}S$ and the right side by $2\alpha+2\beta\,\mathbb{E}X$, so the mean condition makes the inequality invariant under affine changes. The Weierstrass approximation, the reduction to nondecreasing $w$ with $w(a)=0$, and the vertical quantization are unchanged, and give thresholds $a<\eta_1\le\dots\le\eta_R\le b$. Finally, approximating thresholds from below as in (iii)$\Rightarrow$(ii) of \cref{thm:tails} makes both sides converge, because the events $\{X\ge\eta^{(n)}_q\}$ and $\{S\ge\xi^{(n)}_\ell\}$ decrease to $\{X\ge\eta_q\}$ and $\{S\ge\xi_\ell\}$; a threshold equal to $b$ contributes $\mathbb{P}(X=b)$, which need not vanish.
\end{proof}

For $F=U[0,1]$ the right side is $2\sum_q(1-\eta_q)$ and \cref{prop:generalF} is \cref{thm:tails}. On a fixed grid distinct thresholds do not suffice, as the following remark and example show.

\begin{rem}[Grid bounds and repeated thresholds]\label{rem:multiplicity}
Let $u$ be a potential on $\{0,\dots,k-1\}$ whose increments $d(j)=u(j)-u(j-1)$, $1\le j\le k-1$, are nonnegative integers, and let $u(0)=0$. Then $u(i)=\sum_jd(j)\,\ind{i\ge j}$ is the grid form of the staircase whose threshold $j/k$ has multiplicity $d(j)$: the right-continuous staircase of \cref{lem:tailsgen} with these thresholds is the step function $\tilde u$ of \cref{sec:gridineq}, so for a law on $\{1/k,\dots,(2k-1)/k\}$, \eqref{eq:tailsgen} is $I(u)$ divided by $k$. Every rational potential takes this form after adding an affine function and multiplying by a positive constant, and by \cref{lem:affhull} neither operation changes $I(u)$ on mean-one laws. The finitely many potentials of \cite[Proposition~3.4]{LiLiuYao2026} (\cref{sec:gridineq}) are rational, so by \cref{prop:C} the laws in $\DU$ on $(1/k)\Z$, which charge neither $0$ nor $2$ (\cref{prop:C}), are the laws on $\{1/k,\dots,(2k-1)/k\}$ described by nonnegativity, the mean, and the inequalities \eqref{eq:tailsgen} with thresholds in $\{1/k,\dots,(k-1)/k\}$, repetitions allowed; the facets of $\Ph(k)$ other than $x(t)\ge0$ form the minimal such list. Repetitions are needed (\cref{ex:repeated}); how many facets need them is studied in the companion paper \cite{ChaudharyKhandelwalFacets}.
\end{rem}

\begin{ex}\label{ex:repeated}
Let $F_W$ have masses $(1/12,1/3,1/6,1/6,1/4)$ on $\{1/3,1/2,1,4/3,5/3\}$; its mean is $1$.
\begin{enumerate}[label=(\alph*)]
\item For the thresholds $(1/6,1/3,2/3,2/3)$, $(\xi_1,\dots,\xi_8)=(1/3,1/2,5/6,1,4/3,4/3,5/3,5/3)$. The left side of \eqref{eq:tailsgen} is $\frac{12+11+7+7+2\cdot5+2\cdot3}{12}=\frac{53}{12}$, the right side is $2\bigl(\frac56+\frac23+\frac13+\frac13\bigr)=\frac{52}{12}$, so $F_W\notin\DU$. On the grid $k=6$ this inequality is $I(u)$ for $u=(0,1,2,2,4,4)$, and it defines a facet of $\Ph(6)$: the 25 permutations $\sigma$ with $u(i)+u(\sigma(i))=v(i+\sigma(i))$ for all $i$ have X-rays on which $I(u)$ is an equality, and they span an affine space of dimension $8=\dim\Ph(6)-1$ (checked exactly, Section~S5).
\item $F_W$ satisfies $I(u_J)$ for every step set $J\subseteq\{1,\dots,K-1\}$ and $K\in\{6,12,18,24\}$, that is, \eqref{eq:tailsgen} for all distinct thresholds in $\{1/K,\dots,(K-1)/K\}$. This is a finite check over all $2^{K-1}$ step sets (Section~S5).
\item The distinct thresholds $(1/6,1/3,23/36,2/3)$, that is, $I(u_J)$ with $J=\{6,12,23,24\}$ on the grid $K=36$, exclude $F_W$, with violation $1/36$.
\end{enumerate}
On the grids of spacing $1/6$, $1/12$, $1/18$, and $1/24$, then, distinct thresholds do not exclude $F_W$, while one repeated threshold on the coarsest of them does. Part (c) is an instance of the limit in the proof of \cref{thm:tails}: lowering one copy of the repeated threshold $2/3$ by $\delta$ does not decrease the left side of \eqref{eq:tailsgen} and raises the right side by $2\delta$, so with $\delta=1/36$ the violation is at least $1/12-1/18=1/36$. How fine a grid of distinct thresholds is needed to exclude a general lattice law outside $\DU$ is open (\cref{sec:conclusion}).
\end{ex}

\subsection{Two- and three-point laws need many steps}
A \emph{potential with $r$ unit steps} is $u$ or $-u$, where $u$ is a staircase with $r$ thresholds in $(0,1]$, counted with multiplicity; we say it \emph{excludes} a law if the law violates \eqref{eq:potential} for it. The sign $-$ reduces to the sign $+$ by reflection. Write $\tilde F$ for the law of $2-S$, $S\sim F$; reflection preserves $\DU$.

\begin{lem}[Reflection]\label{lem:reflectsteps}
If $-u$, with $u$ a staircase with $r$ thresholds in $(0,1]$, excludes $F$, then some staircase with at most $r$ thresholds in $(0,1)$ excludes $\tilde F$.
\end{lem}

\begin{proof}
Let $\eta_1,\dots,\eta_r$ be the thresholds of $u$ and put $w(x)=r-u(1-x)=\#\{q:1-\eta_q<x\}$. Then $-u(x)=w(1-x)-r$, and the substitution $x\mapsto1-x$ gives $\bigl((-u)\,\square\,(-u)\bigr)(s)=(w\,\square\,w)(2-s)-2r$ and $\int(-u)=\int w-r$. So the violation of \eqref{eq:potential} for $-u$ at $F$ equals the violation for $w$ at $\tilde F$. Let $\hat w(x)=\#\{q:1-\eta_q\le x\}$. Then $\hat w\ge w$, with equality off the finite set $\{1-\eta_q\}$, so $\hat w\,\square\,\hat w\ge w\,\square\,w$ and $\int\hat w=\int w$; the violation for $\hat w$ at $\tilde F$ is at least that for $w$, hence positive. Finally $\hat w=m+u'$, where $m=\#\{q:\eta_q=1\}$ and $u'$ is the staircase with thresholds $\{1-\eta_q:\eta_q<1\}\subset(0,1)$; the constant $m$ adds $2m$ to both sides of \eqref{eq:potential}, so $u'$ excludes $\tilde F$ (and $u'$ is not empty, since the empty staircase excludes nothing).
\end{proof}

\begin{lem}[Counting]\label{lem:counting}
For every staircase $u$ with thresholds $0<\eta_1\le\dots\le\eta_r\le1$ and every $s\in[0,2]$,
\[
(u\,\square\,u)(s)\,(2-s)\ \le\ 2\int_0^1u .
\]
\end{lem}

\begin{proof}
Let $n=(u\,\square\,u)(s)$; there is nothing to prove if $n=0$, and $n\le2r$. Extend the thresholds by $\eta_q=1$ for all $q>r$. By \cref{lem:tailsgen}, $\xi_n\le s$, that is, $\eta_q+\eta_{q'}\le s$ whenever $q+q'=n+1$ and $q,q'\le r+1$. This holds for all $1\le q\le n$ with $q'=n+1-q$: if $q>r+1$, then $q'<n-r\le r$, and the admissible pair $(r+1,n-r)$ gives $\eta_q+\eta_{q'}=1+\eta_{q'}\le1+\eta_{n-r}\le s$; the case $q'>r+1$ is symmetric. Summing over $q=1,\dots,n$ gives $2\sum_{q=1}^n\eta_q\le ns$. The terms $1-\eta_q$ are nonnegative and vanish for $q>r$, so
\[
2\int_0^1u=2\sum_{q=1}^r(1-\eta_q)\ \ge\ 2\sum_{q=1}^n(1-\eta_q)\ \ge\ 2n-ns.
\]
\end{proof}

Every mean-one law on two points $a<1<a+\beta$ has the form
\[
F_{w,\beta}=w\,\delta_a+(1-w)\,\delta_c,\qquad a=1-(1-w)\beta,\quad c=a+\beta=1+w\beta,\quad 0<w<1,
\]
it is supported in $[0,2]$ if and only if $\max(w,1-w)\,\beta\le1$, and its reflection is $F_{1-w,\beta}$. By \cref{rem:B}, $F_{w,\beta}\in\DU$ if and only if $1/\beta\in\N$.

\begin{thm}[Two-point laws]\label{thm:twopoint}
Let $F_{w,\beta}$ be supported in $[0,2]$, let $L=1/\beta$, and let $r_*=\max(1,\lfloor L\rfloor)$.
\begin{enumerate}[label=(\roman*)]
\item If $r+1\le L$, no potential with at most $r$ unit steps (counted with multiplicity), of either sign, excludes $F_{w,\beta}$.
\item If $L\notin\N$, some staircase with $r_*$ steps excludes $F_{w,\beta}$. For $L>1$, with $r=\lfloor L\rfloor$, the thresholds
\[
\eta_q=\frac a2+\Bigl(q-\frac{r+1}2\Bigr)\beta\quad(2\le q\le r),\qquad \eta_1=\min\Bigl(\frac a2-\frac{r-1}2\,\beta,\ w\beta\Bigr)
\]
give the violation $\min\{(1-w)(L-r),\,w(r+1-L)\}/L$ or more; for $L<1$ the single threshold $\eta=w(1+\beta)/2$ gives a violation of at least $w(\beta-1)$.
\end{enumerate}
For $L\notin\N$, therefore, the least number of unit steps, counted with multiplicity, of a potential that excludes $F_{w,\beta}$ is $\max(1,\lfloor 1/\beta\rfloor)$.
\end{thm}

\begin{proof}
(i) By \cref{lem:reflectsteps}, and because the reflection of $F_{w,\beta}$ is $F_{1-w,\beta}$, it suffices to show that no staircase $u$ with $r'\le r$ thresholds in $(0,1]$ excludes $F_{w,\beta}$, for every $w$. Put $A=(u\,\square\,u)(a)$, $C=(u\,\square\,u)(c)$, and $d=C-A$, an integer with $d\ge0$ because $u\,\square\,u$ is nondecreasing. The violation is $V=wA+(1-w)C-2\int u$. By \cref{lem:tailsgen}, $A\ge r'+1$ would need $\xi_{r'+1}\le a$, but $\xi_{r'+1}\ge\eta_{r'+1}+\eta_1=1+\eta_1>1>a$; so $A\le r'\le r$. \Cref{lem:counting} at $s=a$ and at $s=c$ gives $2\int u\ge A(2-a)=A\bigl(1+(1-w)\beta\bigr)$ and $2\int u\ge C(2-c)=C(1-w\beta)$. Hence
\[
V\le(1-w)(d-A\beta)\qquad\text{and}\qquad V\le w(C\beta-d)=w\bigl(A\beta-d(1-\beta)\bigr).
\]
So $V>0$ would force $A\beta<d<A\beta/(1-\beta)$. But $A\beta/(1-\beta)\le r\beta/(1-\beta)\le1$, because $(r+1)\beta\le1$; so $0<d<1$, which is impossible for an integer.

(ii) Case $L>1$. Here $r\beta<1<(r+1)\beta$. Write $\eta_1^\circ=\frac a2-\frac{r-1}2\beta=\frac{1-(r-w)\beta}2$ for the uncapped value, so that $\eta_q=\frac a2+(q-\frac{r+1}2)\beta$ for all $q$ when $\eta_1=\eta_1^\circ$. Then $\eta_1^\circ>0$ since $r\beta<1$, $\eta_1>0$, consecutive uncapped thresholds differ by $\beta$, and $\eta_r=\frac{1+(r-2+w)\beta}2<1$ since $(r-2+w)\beta<r\beta<1$; so the thresholds are valid. Uncapped thresholds satisfy $\eta_q+\eta_{q'}=a+(q+q'-r-1)\beta$. Every pair with $q+q'=r+1$ and $q,q'\le r$ therefore has sum at most $a$ (capping $\eta_1$ only lowers it), so $\xi_r\le a$. Every pair with $q+q'=r+2$ and $2\le q,q'\le r$ has sum $a+\beta=c$, and the pair $(1,r+1)$ has sum $\eta_1+1\le1+w\beta=c$; so $\xi_{r+1}\le c$. Hence $(u\,\square\,u)(a)\ge r$ and $(u\,\square\,u)(c)\ge r+1$. The uncapped thresholds sum to $ra/2$, and capping lowers the sum by $\delta=\max\bigl(0,\eta_1^\circ-w\beta\bigr)=\max\bigl(0,\frac{1-(r+w)\beta}2\bigr)$. So
\[
\begin{aligned}
V&\ge wr+(1-w)(r+1)-2r+ra-2\delta=(1-w)(1-r\beta)-2\delta\\
&=\min\bigl\{(1-w)(1-r\beta),\ w\bigl((r+1)\beta-1\bigr)\bigr\},
\end{aligned}
\]
where the last equality uses $(1-w)(1-r\beta)-(1-(r+w)\beta)=w((r+1)\beta-1)$. Both terms are positive, and they equal $(1-w)(L-r)/L$ and $w(r+1-L)/L$.

Case $L<1$, that is, $\beta>1$. The support condition gives $w\beta\le1$, so $0<\eta=w(1+\beta)/2<1$. Moreover $2\eta=w+w\beta\le1+w\beta=c$, and $1+\eta\le c$ because $\eta\le w\beta$ when $\beta\ge1$. For a single step $\xi_1=2\eta$ and $\xi_2=1+\eta$, so $(u\,\square\,u)(c)=2$ and $V\ge2(1-w)-2(1-\eta)=w(\beta-1)>0$.

The last statement follows: for $L>1$, part (i) with $r=\lfloor L\rfloor-1$, which satisfies $r+1\le L$, shows that fewer than $\lfloor L\rfloor$ steps never suffice, and part (ii) shows that $\lfloor L\rfloor$ steps do; for $L<1$, part (ii) gives one step.
\end{proof}

For $1<L<2$ the potential of (ii) is the single step of \cref{prop:mwwsteps}(a).

\begin{thm}[Symmetric three-point laws]\label{thm:threepoint}
Let $0<b<1$, $L=1/b$, and for $0\le t\le1$ let $G_t$ have masses $\bigl(\frac{1-t}2,t,\frac{1-t}2\bigr)$ on $\{1-b,1,1+b\}$. By \cref{cor:W3}(i), $G_t\in\DU$ if and only if $t\ge\operatorname{dist}(L,2\Z)/L$.
\begin{enumerate}[label=(\roman*)]
\item Let $r\ge1$ and $2r\le L\le2r+2$. The staircase with thresholds
\[
\eta_1=\min\Bigl(b,\ \frac{1-(2r-1)b}2\Bigr),\qquad \eta_q=\frac{1-(2r+3)b}2+2qb\quad(2\le q\le r)
\]
has $(u\,\square\,u)(1-b)=(u\,\square\,u)(1)=r$ and $(u\,\square\,u)(1+b)=r+1$, and \eqref{eq:potential} for it at $G_t$ reads $t\ge\operatorname{dist}(L,2\Z)/L$.
\item If $L\ge2r+2$, no potential with at most $r$ unit steps (counted with multiplicity), of either sign, excludes any $G_t$.
\item If $L\notin2\N$, the least $r$ such that potentials with at most $r$ unit steps exclude every $G_t\notin\DU$ is $\max(1,\lfloor L/2\rfloor)$, and potentials with fewer steps exclude none of them.
\end{enumerate}
\end{thm}

\begin{proof}
(i) \emph{Validity.} $\eta_1>0$ because $L>2r-1$. For $r\ge2$, $\eta_2=\frac{1-(2r-5)b}2>b\ge\eta_1$ because $L>2r-3$, consecutive thresholds from $\eta_2$ on differ by $2b$, and $\eta_r=\frac{1+(2r-3)b}2<1$ because $L>2r-3$. \emph{Values of $u\,\square\,u$.} For $2\le q,q'\le r$, $\eta_q+\eta_{q'}=1-(2r+3)b+2(q+q')b$, which is $1-b$ if $q+q'=r+1$ and $1+b$ if $q+q'=r+2$. The pairs containing $\eta_1$ satisfy $\eta_1+\eta_r\le\frac{1-(2r-1)b}2+\frac{1+(2r-3)b}2=1-b$ (for $r=1$: $2\eta_1\le1-b$) and $\eta_1+\eta_{r+1}=\eta_1+1\le1+b$. So $\xi_r\le1-b$ and $\xi_{r+1}\le1+b$. On the other hand $\xi_{r+1}\ge\eta_1+\eta_{r+1}=1+\eta_1>1$, and for $r\ge2$, $\xi_{r+2}\ge\eta_2+\eta_{r+1}=1+\eta_2>1+b$; for $r=1$, $u\,\square\,u\le2r=2$. With \cref{lem:tailsgen} and monotonicity this gives the values $r,r,r+1$. \emph{The inequality.} It reads $\frac{1-t}2r+tr+\frac{1-t}2(r+1)\le2r-2\sum_q\eta_q$, that is, $t\ge4\sum_q\eta_q-(2r-1)$. Since $\sum_{q=2}^r\eta_q=(r-1)\frac{1+b}2$, the right side is $4\eta_1+2(r-1)b-1$. If $L\ge2r+1$, then $\eta_1=b$ and it equals $(2r+2)b-1=(2r+2-L)/L$; if $L\le2r+1$, then $\eta_1=\frac{1-(2r-1)b}2$ and it equals $1-2rb=(L-2r)/L$. In both ranges this is $\operatorname{dist}(L,2\Z)/L$.

(ii) $G_t=(1-t)\,F_{1/2,2b}+t\,\delta_1$, where $F_{1/2,2b}=\frac12(\delta_{1-b}+\delta_{1+b})$ is the two-point law of spacing $\beta=2b$. Since $1/\beta=L/2\ge r+1$, \cref{thm:twopoint}(i) shows that every potential with at most $r$ unit steps holds at $F_{1/2,2b}$, and it holds at $\delta_1\in\DU$; the inequality is linear in the law.

(iii) If $1<L<2$, one step suffices by \cref{prop:mwwsteps}(b), and the potential with no steps is constant. If $L>2$, let $r''=\lfloor L/2\rfloor$, so $2r''<L<2r''+2$. Part (i) with $r''$ excludes every $G_t$ with $t<\operatorname{dist}(L,2\Z)/L$, and part (ii) with $r''-1$ steps, applicable because $L\ge2r''$, excludes none.
\end{proof}

\begin{cor}\label{cor:noatombound}
For every $r\ge1$, the laws on $[0,2]$ with mean one that satisfy \eqref{eq:potential} for every potential with at most $r$ unit steps form a set strictly larger than $\DU$; it contains the two-point laws $F_{w,\beta}$ with $r+1<1/\beta<r+2$. In particular, the number of unit steps, counted with multiplicity, needed to exclude a law outside $\DU$ is not bounded by any function of its number of atoms.
\end{cor}

\begin{proof}
For $r+1<1/\beta<r+2$, $F_{w,\beta}$ is supported in $[0,2]$, lies outside $\DU$ because $1/\beta\notin\N$, and satisfies all these inequalities by \cref{thm:twopoint}(i). \Cref{thm:twopoint} shows that two-point laws need $\max(1,\lfloor1/\beta\rfloor)$ steps, which is unbounded.
\end{proof}

What governs the number of steps is the spacing, not the number of atoms; this matches the grid picture, in which a law on $\{1/k,\dots,(2k-1)/k\}$ is decided by thresholds in $\{1/k,\dots,(k-1)/k\}$ (\cref{rem:multiplicity}). Asymmetric three-point laws behave the same way. The mean-one law with masses $(1/160,1/40,31/32)$ on $\{123/280,29/40,283/280\}$ is not in $\DU$, satisfies every bound with at most two unit steps, and is excluded by a staircase with three (Proposition~S4.1).

\section{Laws with a bounded density}\label{sec:density}

Every example so far that separates $\DU$ from the convex-order set has a large atom or a large density. Bounding the density brings convex order back, up to a point. It decides membership at density one; it fails above density two; and for single steps the cutoff is exactly two. The range between one and two holds the most interesting open question of the paper (\cref{prob:density}).

For a random variable $X$ with finite mean write $\pi_X(x)=\mathbb{E}(X-x)^+$ for its stop-loss function.

\begin{thm}[Densities at most one]\label{thm:density}
Let $S$ have a Lebesgue density $g\le1$. Then $\law(S)\in\DU$ if and only if $S\le_{\mathrm{cx}}2U$.
\end{thm}

For such laws the convex potentials are the only ones needed in \cref{thm:dual}. The proof reduces to the lattice statement \cref{cor:smallatoms}.

\begin{proof}
Necessity is \cref{cor:readings}(i). For sufficiency let $S\le_{\mathrm{cx}}2U$; then $S\in[0,2]$ almost surely and $\mathbb{E}S=1$.

\emph{Step 1: a margin.} Fix $t\in(0,1/4]$, put $\rho=1-t/3$, and let $S'$ have the law $(1-t)\law\bigl(1+\rho(S-1)\bigr)+t\,U[1/4,7/4]$. Its density is at most $(1-t)/\rho+2t/3\le1$, because $(1-t)/(1-t/3)\le1-2t/3$ (equivalent to $0\le 2t^2/9$). It has mean $1$ and is supported in $[t/3,2-t/3]$. Since $1+\rho(S-1)\le_{\mathrm{cx}}1+\rho(2U-1)$, which is uniform on $[1-\rho,1+\rho]$, and uniform laws on intervals symmetric about $1$ are increasing in convex order as the interval grows,
\[
\pi_{2U}(x)-\pi_{S'}(x)\ \ge\ t\bigl(\pi_{2U}(x)-\pi_{V}(x)\bigr),\qquad V\sim U[1/4,7/4].
\]
Here $\pi_{2U}(x)=(2-x)^2/4$ on $[0,2]$. For $|x-1|\le3/4$ the bracket equals $\frac1{12}\bigl(\frac34-(x-1)^2\bigr)\ge\frac1{64}$; for $x\in[0,1/4]$ it equals $x^2/4$, and symmetrically $(2-x)^2/4$ on $[7/4,2]$. So on $[t/3,2-t/3]$ the bracket is at least $t^2/36$, and the gap $\pi_{2U}-\pi_{S'}$ is at least $t^3/36$. For $x<t/3$ the gap is $x^2/4$, since $\pi_{S'}(x)=1-x$ there, and symmetrically it is $(2-x)^2/4$ for $x>2-t/3$. Hence, for every integer $k\ge3t^{-3/2}$ and every integer $1\le m\le2k-1$,
\begin{equation}\label{eq:margin}
\pi_{2U}(m/k)-\pi_{S'}(m/k)\ \ge\ \frac1{4k^2}.
\end{equation}

\emph{Step 2: rounding.} Let $S'_k$ be the randomized rounding of $S'$ to $(1/k)\Z$: given $S'$, it equals $\lfloor kS'\rfloor/k$ with probability $1-\frc(kS')$ and $\lceil kS'\rceil/k$ otherwise. Then $\mathbb{E}[S'_k\mid S']=S'$, and the atom of $S'_k$ at $j/k$ is $\int g'(s)\,(1-|ks-j|)^+\,ds\le1/k$, where $g'\le1$ is the density of $S'$. Since $k\ge3t^{-3/2}\ge3/t$, we have $1/k\le t/3$, so $S'_k$ has no atom at $0$ or $2$: it is supported in $\{1/k,\dots,(2k-1)/k\}$. Since $x\mapsto(x-m/k)^+$ is linear between consecutive lattice points, $\pi_{S'_k}(m/k)=\pi_{S'}(m/k)$ for every integer $m$.

\emph{Step 3: lattice convex order.} Let $V_k$ be uniform on $\{1/k,3/k,\dots,(2k-1)/k\}$, so that $kV_k-1=2U_k$. Summing directly, $\pi_{V_k}(m/k)=\pi_{2U}(m/k)$ for even $m$ and $\pi_{V_k}(m/k)=\pi_{2U}(m/k)-\frac1{4k^2}$ for odd $m$, $0\le m\le2k$. By \eqref{eq:margin} and Step 2, $\pi_{S'_k}(m/k)\le\pi_{V_k}(m/k)$ for $1\le m\le2k-1$; both sides equal $1-m/k$ for $m\le0$ and $0$ for $m\ge2k$. Both functions are piecewise linear with kinks only at lattice points, so $\pi_{S'_k}\le\pi_{V_k}$ everywhere, that is, $kS'_k-1\le_{\mathrm{cx}}2U_k$. All atoms of $S'_k$ are at most $1/k$ and its mean is one, so $\law(S'_k)\in\DU$ by \cref{cor:smallatoms}.

\emph{Step 4: limits.} $|S'_k-S'|\le1/k$, so $S'_k\to S'$ in law, and $\law(S')\in\DU$ because $\DU$ is weakly closed (\cref{sec:setup}). As $t\to0$, $\law(S')\to\law(S)$ weakly, so $\law(S)\in\DU$.
\end{proof}

For example, the law with density $1$ on $[1/3,5/6]\cup[7/6,5/3]$ satisfies $S\le_{\mathrm{cx}}2U$ and therefore belongs to $\DU$; it is not unimodal, so the unimodal criterion of \cite{MaoWangWang2019} does not apply to it. The constant $1$ in \cref{thm:density} cannot be replaced by any constant above $2$, even if the folded condition of \cref{cor:readings}(ii) is added.

\begin{prop}\label{prop:densitysharp}
For $c>2$ let $\eta=\frac{2c-1}{6c}$, $h=\frac1{2c}$, and let $S$ have density $c$ on $[2\eta,2\eta+h]\cup[1+\eta,1+\eta+h]$. Then $\mathbb{E}S=1$, $S\le_{\mathrm{cx}}2U$, and the folded condition $\law(U)\le_{\mathrm{cx}}\frac12\bigl(\law(S^\flat)+\law(\kappa)\bigr)$ of \cref{cor:readings}(ii) holds, but the single step at $\eta$ is violated by $\frac{c-2}{6c}$. So $\law(S)\notin\DU$.
\end{prop}

\begin{proof}
Each interval carries mass $ch=1/2$, and $2\eta+h=\frac{4c+1}{6c}<1<1+\eta$, $1+\eta+h=1+\frac{2c+2}{6c}<2$. \emph{Mean:} $\frac12(2\eta+\frac h2)+\frac12(1+\eta+\frac h2)=\frac{1+3\eta+h}2=1$. \emph{One step:} $\mathbb{P}(S\ge2\eta)+\mathbb{P}(S\ge1+\eta)=\frac32$, while $2(1-\eta)=\frac{4c+1}{3c}$; the difference is $\frac{c-2}{6c}$.

\emph{Convex order.} Let $D=\pi_{2U}-\pi_S$. Left of the first interval $D(x)=x^2/4$, and right of the second $D(x)=(2-x)^2/4$. Between the intervals, $D(x)=\frac{(2-x)^2}4-\frac12\bigl(1+\eta+\frac h2-x\bigr)$, with $D'(x)=\frac{x-1}2$; its minimum on $[2\eta+h,1+\eta]$ is at $x=1$, where $D(1)=\frac{2c-1}{24c}>0$. On each interval $D''=\frac12-c<0$, so $D$ is concave there and is minimized at an endpoint. Hence $D\ge0$.

\emph{Folded order.} With $\kappa=\ind{S>1}$, the law $M=\frac12\bigl(\law(S^\flat)+\law(\kappa)\bigr)$ has atoms $\frac14$ at $0$ and at $1$ and density $c/2$ on $[\eta,\eta+h]$ and on $[2\eta,2\eta+h]$; both $M$ and $U$ have mean $\frac12$. Put $H(x)=\pi_M(x)-\pi_U(x)=\int_0^x(F_M(y)-y)\,dy$ for $x\in[0,1]$, where $F_M$ is the distribution function of $M$; $H(0)=H(1)=0$. Because $c>2$, we have $\frac14\le\eta<\eta+h\le\frac12\le2\eta<2\eta+h\le\frac34$, so $F_M$ equals $\frac14,\frac12,\frac34$ on the gaps $[0,\eta)$, $[\eta+h,2\eta)$, $[2\eta+h,1)$ and has slope $c/2>1$ on the two intervals, where it crosses the diagonal at $y=\frac13$ and $y=\frac23$. So $F_M(y)-y$ changes sign from $+$ to $-$ at $\frac14,\frac12,\frac34$ and from $-$ to $+$ at $\frac13,\frac23$, and the interior local minima of $H$ are at $\frac13$ and $\frac23$. Direct integration gives $H(\frac13)=H(\frac23)=\frac{4c+1}{144c}>0$. Hence $H\ge0$ on $[0,1]$, which is the folded condition.
\end{proof}

At $c=2$ the same construction gives violation $0$, and the law is then attainable (\cref{prop:coupling}). For single steps the constant $2$ is exactly right, as we now show. The argument uses a dual form of the density condition.

\begin{lem}[Dual certificates]\label{lem:densitydual}
Let $c>0$, let $u$ be a bounded potential on $[0,1]$ such that $N=u\,\square\,u$ is Borel, and let $\varphi:[0,2]\to\R$ be convex. If
\[
\frac12\int_0^2\varphi(s)\,ds+c\int_0^2\bigl(N(s)-\varphi(s)\bigr)^+\,ds\ \le\ 2\int_0^1u,
\]
then \eqref{eq:potential} holds for $u$ at the law of every $S$ that has a density at most $c$ and satisfies $S\le_{\mathrm{cx}}2U$.
\end{lem}

\begin{proof}
A convex function on $[0,2]$ is continuous on $(0,2)$ and has finite limits at $0$ and $2$; changing $\varphi$ at the two endpoints affects none of the integrals below, because the laws involved have densities. So we may assume $\varphi$ continuous. Its piecewise linear interpolants $\varphi_n$ at the points $2j/n$ are convex, extend to convex functions on $\R$ by continuing their end pieces, and converge to $\varphi$ uniformly on $[0,2]$. Since $S\le_{\mathrm{cx}}2U$, $\mathbb{E}\varphi_n(S)\le\mathbb{E}\varphi_n(2U)$ for every $n$, hence $\mathbb{E}\varphi(S)\le\mathbb{E}\varphi(2U)=\frac12\int_0^2\varphi$. Let $g\le c$ be the density of $S$. From $N\le\varphi+(N-\varphi)^+$,
\[
\mathbb{E}N(S)\ \le\ \mathbb{E}\varphi(S)+\int_0^2(N-\varphi)^+g\ \le\ \frac12\int_0^2\varphi+c\int_0^2(N-\varphi)^+\ \le\ 2\int_0^1u.
\]
\end{proof}

\begin{prop}[Single steps at density two]\label{prop:singlestep}
Let $S$ have a density at most $2$ and satisfy $S\le_{\mathrm{cx}}2U$. For $0<\eta<1$ put $t=(8\eta-2)/3$ and $t^-=\max(0,-t)$. Then
\[
\mathbb{P}(S\ge2\eta)+\mathbb{P}(S\ge1+\eta)\ \le\ 2(1-\eta)-\frac{(t^-)^2}{4(1-\eta)}.
\]
In particular $S$ satisfies every single-step inequality \eqref{eq:onestep}. For $\eta\le1/4$ the bound holds for every $S$ with values in $[0,2]$, mean one, and a density at most $2$.
\end{prop}

\begin{proof}
Here $N(s)=\ind{s\ge2\eta}+\ind{s\ge1+\eta}$ on $[0,2]$ and $2\int u=2(1-\eta)$ (\cref{cor:readings}(iii)). Apply \cref{lem:densitydual} with $c=2$ and $\varphi(s)=(s-t)^+/(1-\eta)$. Since $t\le2\eta$, $\varphi\ge0=N$ on $[0,2\eta)$. The function $\varphi$ has slope $1/(1-\eta)$ on $[\max(t,0),2]$, with $\varphi(2\eta)=\frac23$, $\varphi\bigl(\frac{5\eta+1}3\bigr)=1$, $\varphi(1+\eta)=\frac53$, and $\varphi\bigl(\frac{2\eta+4}3\bigr)=2$, and $2\eta\le\frac{5\eta+1}3\le1+\eta\le\frac{2\eta+4}3\le2$. So $(N-\varphi)^+$ consists of two triangles, on $\bigl[2\eta,\frac{5\eta+1}3\bigr]$ and on $\bigl[1+\eta,\frac{2\eta+4}3\bigr]$, each with base $\frac{1-\eta}3$ and height $\frac13$, and $2\int_0^2(N-\varphi)^+=\frac{2(1-\eta)}9$. If $t\ge0$, then $\frac12\int_0^2\varphi=\frac{(2-t)^2}{4(1-\eta)}=\frac{16(1-\eta)}9$, because $2-t=\frac{8(1-\eta)}3$, and the hypothesis of \cref{lem:densitydual} holds with equality. If $t<0$, then $\varphi$ is affine on $[0,2]$, and $\frac12\int_0^2\varphi=\frac{1-t}{1-\eta}=\frac{(2-t)^2}{4(1-\eta)}-\frac{t^2}{4(1-\eta)}$; the chain of inequalities in the proof of \cref{lem:densitydual} gives $\mathbb{E}N(S)\le\frac12\int_0^2\varphi+2\int_0^2(N-\varphi)^+=2(1-\eta)-\frac{t^2}{4(1-\eta)}$, which is the claim. For $t\le0$, that is, $\eta\le1/4$, $\varphi$ is affine, so $\mathbb{E}\varphi(S)=\frac12\int_0^2\varphi$ follows from $\mathbb{E}S=1$ alone, and convex order is not needed.
\end{proof}

\Cref{prop:singlestep} is the case $r=1$, $c=2$ of the following certificate. For thresholds $0<\eta_1<\dots<\eta_r<1$, with $\eta_{r+1}=1$ as before, let $N=u\,\square\,u$ for the staircase $u$ with these thresholds, $B=2\sum_q(1-\eta_q)$, $d_j=\eta_{j+1}-\eta_j$ for $1\le j\le r$, and $D=1-\eta_1=\sum_jd_j$.

\begin{thm}[Nonincreasing threshold gaps]\label{thm:gaps}
Suppose $d_1\ge d_2\ge\dots\ge d_r$ (the last gap $d_r=1-\eta_r$ included) and $\frac12\le c\le1+D/d_1$; since $D\ge d_1$, this allows $c=2$. Then \eqref{eq:tailsgen} holds for every $S$ with a density at most $c$ and $S\le_{\mathrm{cx}}2U$. More precisely, with $\alpha=\frac{2D(2c-1)}{4cD+d_1}$ and $t=2\eta_1-\alpha d_1$, the convex function
\[
\varphi(s)=\frac{(s-t)^+}{d_1}+\sum_{j=2}^{r}\Bigl(\frac1{d_j}-\frac1{d_{j-1}}\Bigr)(s-2\eta_j)^+
\]
satisfies
\[
B-\frac12\int_0^2\varphi-c\int_0^2(N-\varphi)^+=\frac{D\,[D+(1-c)d_1]}{4cD+d_1}+\frac{\bigl((-t)^+\bigr)^2}{4d_1}\ \ge\ 0,
\]
which is positive for $c=2$ and $r\ge2$.
\end{thm}

\begin{proof}
If $q<q'-1$, then $(\eta_{q+1}+\eta_{q'-1})-(\eta_q+\eta_{q'})=d_q-d_{q'-1}\ge0$, so the maxima defining $\xi_\ell$ are attained by balanced pairs: $\xi_{2j-1}=2\eta_j$ and $\xi_{2j}=\eta_j+\eta_{j+1}$, and the two plateaus of $N$ between $2\eta_j$ and $2\eta_{j+1}$ both have length $d_j$. The coefficients of $\varphi$ are nonnegative, $\varphi$ has slope $1/d_j$ on $[2\eta_j,2\eta_{j+1}]$, and $\varphi(2\eta_j)=2j-2+\alpha$, $\varphi(\eta_j+\eta_{j+1})=2j-1+\alpha$. For $c\ge\frac12$ we have $0\le\alpha<1$. So after each jump $N-\varphi$ starts at $1-\alpha$ and decreases with slope $1/d_j$, it vanishes within the plateau, and $\int_0^2(N-\varphi)^+=\sum_j2\cdot\frac12(1-\alpha)^2d_j=D(1-\alpha)^2$; below $2\eta_1$, $N=0\le\varphi$. With $L_j=1-\eta_j$, so that $L_j-L_{j+1}=d_j$, integrating the hinges gives
\[
\frac12\int_0^2\varphi=\frac{(2D+\alpha d_1)^2}{4d_1}+\sum_{j=2}^r\Bigl(\frac1{d_j}-\frac1{d_{j-1}}\Bigr)L_j^2-\frac{((-t)^+)^2}{4d_1}=B-D+D\alpha+\frac{d_1\alpha^2}4-\frac{((-t)^+)^2}{4d_1},
\]
by the telescoping identity $\frac{D^2}{d_1}+\sum_{j\ge2}(\frac1{d_j}-\frac1{d_{j-1}})L_j^2=\sum_j\frac{L_j^2-L_{j+1}^2}{d_j}=\sum_j(L_j+L_{j+1})=B-D$. The margin is therefore $D(1-c)+D(2c-1)\alpha-(cD+\frac{d_1}4)\alpha^2+\frac{((-t)^+)^2}{4d_1}$, the chosen $\alpha$ maximizes the sum of the first three terms, and substituting it gives the stated value (the last term is nonnegative; it depends on $\alpha$, so this $\alpha$ need not maximize the whole margin). It is nonnegative for $c\le1+D/d_1$, and positive for $c=2$ when $D>d_1$, that is, $r\ge2$. \Cref{lem:densitydual} finishes the proof.
\end{proof}

For equally spaced thresholds the cap of \cref{thm:gaps} is exact, and it grows with the number of steps; \cref{prop:densitysharp} is the case $r=1$.

\begin{thm}[Equally spaced thresholds]\label{thm:equal}
Let $r\ge1$, $0<d<1/r$ and $\eta_q=1-(r+1-q)d$ for $1\le q\le r$. If $\frac12\le c\le r+1$, then every $S$ with a density at most $c$ and $S\le_{\mathrm{cx}}2U$ satisfies \eqref{eq:tailsgen} for these thresholds. For every $c>r+1$, with $d=\frac{4cr+1}{2cr(2r+1)}$ and $\delta=\frac1{2cr}$, the law with density $c$ on the $2r$ intervals $[\xi_\ell,\xi_\ell+\delta]$ has mean one and is dominated by $2U$ in convex order, and it violates \eqref{eq:tailsgen} by $\frac{c-r-1}{2c(2r+1)}$.
\end{thm}

\begin{proof}
All gaps equal $d$, so the first statement is \cref{thm:gaps} with $D/d_1=r$. Let $c>r+1$. Here $\xi_\ell=2-(2r+1-\ell)d$, and $(2r+1)d=2+\delta$; both $rd<1$ and $\delta<d$ are equivalent to $c>\frac12$. So the intervals are disjoint subintervals of $(0,2)$, each of mass $c\delta=1/(2r)$, and $N=\ell$ almost everywhere on the $\ell$-th of them. Let $\Lambda$ be uniform on $\{1,\dots,2r\}$ and $V$ uniform on $[-\frac12,\frac12]$, independent. Then $S=1+d(\Lambda-\frac{2r+1}2)+\delta V$ has the stated law, while $1+\frac1r(\Lambda-\frac{2r+1}2)+\frac1rV$ is uniform on $[0,2]$. Since $aX\le_{\mathrm{cx}}X$ for centered $X$ and $0\le a\le1$, and convex order is preserved under adding independent variables, the contraction factors $rd<1$ and $r\delta=\frac1{2c}<1$ give $S\le_{\mathrm{cx}}2U$. Finally $\mathbb{E}N(S)=\frac{2r+1}2$ and $B=r(r+1)d$, whose difference is $\frac{c-r-1}{2c(2r+1)}$.
\end{proof}

Since the margin of \cref{thm:gaps} at $c=2$ is positive for $r\ge2$, the same $\varphi$ also certifies every threshold tuple close enough to an equally spaced one (moving each threshold by at most $\varepsilon$ changes $B$ and the certificate value by at most $10r\varepsilon$ in total). For two steps, \cref{thm:gaps} covers only $b-a\ge1-b$; the following theorem covers all of them.

\begin{thm}[Two steps at density two]\label{thm:twostep}
Let $0<a<b<1$. Every $S$ with a density at most $2$ and $S\le_{\mathrm{cx}}2U$ satisfies
\[
\begin{aligned}
&\mathbb{P}(S\ge2a)+\mathbb{P}(S\ge a+b)+\mathbb{P}\bigl(S\ge\max(1+a,2b)\bigr)+\mathbb{P}(S\ge1+b)\\
&\qquad\le\ 2(1-a)+2(1-b)-\frac{\min(b-a,\,1-b)}{21}.
\end{aligned}
\]
\end{thm}

The proof (Section~S4.2) splits by the ratio $\rho=(b-a)/(1-b)$ of the two gaps and writes down an explicit convex certificate for each of three ranges of $\rho$. The constant $\frac1{21}$ is the smallest margin these certificates give, reached at $\rho=\frac34$. The law of \cref{prop:densitysharp} at $c=2$ is attainable, so the single-step bound of \cref{prop:singlestep} holds with equality at $\eta=\frac14$.

\begin{prop}[An attainable boundary law]\label{prop:coupling}
Let $X$ be uniform on $[0,1]$ and, given $X=x$, let $Y=\frac34-2x$ if $x\le\frac14$; $Y=\frac38-\frac x2$ or $Y=\frac98-\frac x2$, with probability $\frac12$ each, if $\frac14<x\le\frac34$; and $Y=\frac94-2x$ if $x>\frac34$. Then $Y$ is uniform on $[0,1]$, and $X+Y$ has density $2$ on $[\frac12,\frac34]\cup[\frac54,\frac32]$, which is the law of \cref{prop:densitysharp} with $c=2$ ($\eta=h=\frac14$). It satisfies $\mathbb{P}(S\ge\frac12)+\mathbb{P}(S\ge\frac54)=\frac32=2(1-\frac14)$.
\end{prop}

\begin{proof}
A branch $y=\beta+\gamma x$ taken with conditional probability $\pi$ contributes density $\pi/|\gamma|$ on its image. The four branches give $Y$ the densities $\frac12$ on $[\frac14,\frac34]$, $1$ on $[0,\frac14]$, $1$ on $[\frac34,1]$, and $\frac12$ on $[\frac14,\frac34]$, which add up to $1$ on $[0,1]$; and they give $X+Y=\frac34-x$, $\frac38+\frac x2$, $\frac98+\frac x2$, $\frac94-x$ the density $1$ each on $[\frac12,\frac34]$, $[\frac12,\frac34]$, $[\frac54,\frac32]$, $[\frac54,\frac32]$.
\end{proof}

With \cref{prop:densitysharp}, the largest constant $c$ such that a density at most $c$ together with $S\le_{\mathrm{cx}}2U$ implies every single-step inequality \eqref{eq:onestep} is exactly $2$. Whether these two conditions decide membership is open.

\begin{prob}\label{prob:density}
Let $S$ have a density at most $2$ and satisfy $S\le_{\mathrm{cx}}2U$. Is $\law(S)\in\DU$?
\end{prob}

The certificates of \cref{lem:densitydual} are complete, and \cref{prob:density} reduces to extreme densities, as follows. Call a law \emph{attainable} if it belongs to $\DU$.

\begin{prop}[Duality and reduction]\label{prop:duality}
Let $c\ge\frac12$, let $\mathcal{H}_c$ be the set of densities $g$ on $[0,2]$ with $0\le g\le c$ and $g(s)\,ds\le_{\mathrm{cx}}U[0,2]$, and for $g\in\mathcal{H}_c$ let $D_g(t)=\frac{(2-t)^2}4-\int_t^2(s-t)g(s)\,ds\ge0$.
\begin{enumerate}[label=(\alph*)]
\item For bounded measurable $N$ on $[0,2]$, $\max_{g\in\mathcal{H}_c}\int Ng=\min_\varphi\bigl[\frac12\int_0^2\varphi+c\int_0^2(N-\varphi)^+\bigr]$, the minimum over continuous convex $\varphi$ on $[0,2]$; both are attained.
\item Feasible $g$ and $\varphi$ are optimal if and only if $\int_{(0,2)}D_g\,d\varphi''=0$, $g=c$ a.e.\ on $\{N>\varphi\}$, and $g=0$ a.e.\ on $\{N<\varphi\}$. In particular $\varphi$ is affine on each component of $\{D_g>0\}$.
\item $g=\frac12$ a.e.\ on $Z_g=\{D_g=0\}$. For each component $(s_0,s_1)$ of $[0,2]\setminus Z_g$, with $p=(s_1-s_0)/2$, the density $\tilde g(t)=g(s_0+pt)$ lies in $\mathcal{H}_c$, and $D_{\tilde g}(t)=p^{-2}D_g(s_0+pt)>0$ for $0<t<2$. If all these laws $\tilde g(t)\,dt$ are attainable, so is $g(s)\,ds$.
\item $g$ is an extreme point of $\mathcal{H}_c$ if and only if $g\in\{0,c\}$ a.e.\ on $\{D_g>0\}$. Consequently every law in $\mathcal{H}_c$ is attainable if and only if every $g\in\mathcal{H}_c$ with $g\in\{0,c\}$ a.e.\ and $D_g>0$ on $(0,2)$ is attainable.
\end{enumerate}
\end{prop}

The proof, a minimax argument followed by a Krein--Milman reduction, is in Section~S4.3.

\Cref{prob:density} therefore reduces to bang-bang laws: densities with values in $\{0,2\}$ whose stop-loss gap $D_g$ is positive on $(0,2)$. Part (a) also says what a positive answer would take. By \cref{thm:tails}(iii) only distinct thresholds matter, and a staircase inequality holds on $\mathcal{G}_2$ exactly when some convex $\varphi$ satisfies the hypothesis of \cref{lem:densitydual}. \Cref{prop:singlestep,thm:gaps,thm:equal,thm:twostep} build such certificates by hand, and on the grids of spacing $1/8$, $1/12$, and $1/16$ there is one for every staircase with distinct thresholds (Section~S5). What is missing is a uniform construction. The obvious candidate, summing the single-step certificates, already fails for five thresholds (Section~S5).

\section{Three-atom laws and single steps}\label{sec:windows}

This section applies the lattice reduction (\cref{thm:B}) to the three-point laws studied by Mao et al.\ \cite{MaoWangWang2019}, and shows that in a range of spacings their obstructions are single-step bounds. Equidistant four-atom laws are described completely in Section~S3, through a link between one-sided windows of displacements and Gomory's master knapsack polytope.

Reverse the columns: with $\tau=(m-1)-\sigma$, the sum $i+\sigma(i)$ equals $(m-1)+(i-\tau(i))$, so a point of $P(m)$ is a histogram of displacements $i-\tau(i)$. If the histogram vanishes outside $\{-1,0,1\}$, so does every entry of the matrix off that band, because entries are nonnegative; and the permutations on the band are the products of disjoint adjacent transpositions. It follows that the points of $P(m)$ on the displacements $\{-1,0,1\}$ are $\{(s,1-2s,s):0\le s\le\lfloor m/2\rfloor/m\}$, and the middle mass is at least $(m\bmod2)/m$.

\begin{cor}[Three equidistant atoms]\label{cor:W3}
Let $F$ be a law with mean one on $\{a-b,a,a+b\}$, where $0<b\le a\le1$, with masses $(f_1,f_2,f_3)$. Let $L=1/b$, $n=\lfloor L\rfloor$, and $\theta=\frc(L)$ as in \cref{thm:B}.
\begin{enumerate}[label=(\roman*)]
\item If $a=1$ (so that $f_1=f_3$ by the mean), then $F\in\DU$ if and only if $f_2\ge\operatorname{dist}(L,2\Z)/L$.
\item If $a<1$ and $a+b<1$, no law on the support has mean one.
\item If $a<1$ and $a+b=1$, the only mean-one law on the support is $\delta_1$, and it belongs to $\DU$ for every $b$.
\item If $a<1<a+b$, then $F\in\DU$ if and only if $k=1/b\in\N$ and, with $C=(a+b-1)/b=\frc(a/b)\in(0,1)$, $f_2\in[C\,(k\bmod 2)/k,\ C]$; the mean then forces $f_1=(C-f_2)/2$ and $f_3=1-(C+f_2)/2$. For even $k$ the interval is $[0,C]$; for odd $k$ it is $[a+b-1,C]$.
\end{enumerate}
\end{cor}

\begin{proof}
(i) $1$ is on the grid. By \cref{thm:B}, the two fibers have sizes $n+1$ and $n$, with weights $\theta(n+1)/L$ and $(1-\theta)n/L$, and both are centered at the mean. On a width-three window each fiber of size $m$ contributes a symmetric law whose middle mass fills the interval $[(m\bmod2)/m,\,1]$ (the segment $0\le s\le\lfloor m/2\rfloor/m$ above), so $f_1=f_3$, the attainable middle masses form an interval with right end $1$, and its left end is $\theta(n+1)/L\cdot((n+1)\bmod2)/(n+1)+(1-\theta)n/L\cdot(n\bmod2)/n$, which equals $\operatorname{dist}(L,2\Z)/L$ in both parities of $n$ (it is $\theta/L$ for even $n$ and $(1-\theta)/L$ for odd $n$; when $\theta=0$ the first fiber has weight zero). (ii) is immediate. (iii) $Y=1-X$ gives $\delta_1$. (Here and below the support restriction is imposed fiber by fiber; forbidden masses cannot cancel because all masses are nonnegative.) (iv) Here $1$ lies strictly between two atoms, so it is off the grid and \cref{thm:B}(i) forces $k=1/b\in\N$. The fiber of weight $C$ is $P(k)$, centered on the middle atom; the fiber of weight $1-C$ is $P(k)+1$, whose support would have to lie below its own mean, so it is the point mass at $a+b$. The first fiber contributes a symmetric law with middle mass between $(k\bmod2)/k$ and $1$, which gives the stated interval after scaling by $C$ and the stated $f_1,f_3$. Finally $C(k\bmod2)/k=(a+b-1)$ when $k$ is odd, since $C/k=(a+b-1)$.
\end{proof}

Laws with $a>1$ reduce to $a<1$ by the reflection $S\mapsto2-S$, which preserves $\DU$. The quantity $\operatorname{dist}(L,2\Z)/L$ in (i) is the `non-integrity' of $1/(2b)$ used in \cite{MaoWangWang2019}. \Cref{cor:W3} agrees with their Theorem~4 (arXiv version: Theorem~3.5) for laws with three positive masses once three misprints are corrected, the lower bound $a+b-1$ for odd $k$ among them; see Appendix~\ref{app:mww}.

\subsection{Single steps behind the known obstructions}
In the ranges of \cref{prop:mwwsteps} the obstructions to the two- and three-point laws of \cite{MaoWangWang2019} are single-step inequalities \eqref{eq:onestep}.

\begin{prop}\label{prop:mwwsteps}
\begin{enumerate}[label=(\alph*)]
\item Let $F$ be a mean-one law on $\{a,a+b\}$ with $a<1<a+b$ and $\frac12<b<1$. Then $F$ violates \eqref{eq:onestep} for $\eta=\min(a/2,\,a+b-1)$; in particular $F\notin\DU$.
\item Let $F$ be the law with masses $(f,f_2,f)$ on $\{1-b,1,1+b\}$, where $\frac14\le b<1$, and put $L=1/b$. Take $\eta=b$ if $L\in(1,2]\cup[3,4]$ and $\eta=(1-b)/2$ if $L\in[2,3]$. Then \eqref{eq:onestep} for this $\eta$ reads $f_2\ge\operatorname{dist}(L,2\Z)/L$, so this single inequality decides whether $F\in\DU$.
\end{enumerate}
\end{prop}

\begin{proof}
(a) The masses are $w=(a+b-1)/b$ at $a$ and $1-w=(1-a)/b$ at $a+b$. For $\eta\le a/2$ and $\eta\le a+b-1$ we have $2\eta\le a$ and $1+\eta\le a+b$, so the left side of \eqref{eq:onestep} is $2-w$, which exceeds $2-2\eta$ if and only if $\eta>w/2$. Now $w<a$ if and only if $a+b-1<ab$, that is, $(1-a)(1-b)>0$, which holds, and $w/2<a+b-1$ if and only if $b>\frac12$. So $\eta=\min(a/2,a+b-1)$ lies in $(w/2,1)$.
(b) For $L\in(1,2)$, $2\eta=2b\in(1,1+b]$, so both probabilities equal $f$ and the inequality is $f\le1-b$, that is, $f_2\ge2b-1=(2-L)/L$. For $L\in[2,3]$, $2\eta=1-b$ and $1<1+\eta\le1+b$, so the inequality is $1+f\le1+b$, that is, $f_2\ge1-2b=(L-2)/L$. For $L\in[3,4]$, $2b\le1-b$, so the inequality is $1+f\le2-2b$, that is, $f_2\ge4b-1=(4-L)/L$. (At $L=2$ both choices give the trivial bound $f_2\ge0$.) In each range the right side is $\operatorname{dist}(L,2\Z)/L$, and by \cref{cor:W3}(i) a symmetric law with this property belongs to $\DU$.
\end{proof}

In other words, the symmetric case of the three-point theorem of \cite{MaoWangWang2019} ($a=1$ in \cref{cor:W3}) for spacings $\frac14\le b<1$, and the necessity half of their two-point theorem for $\frac12<b<1$, are statements about single-step bounds. The asymmetric case (iv) of \cref{cor:W3} is not: the law $F_*$ of Proposition~S4.1 satisfies every single-step inequality and needs three unit steps, and smaller spacings need more (\cref{thm:twopoint,thm:threepoint}).
\section{Concluding remarks}\label{sec:conclusion}

For two uniforms the right replacement for convex order is a family of joint tail bounds, one for each finite set of thresholds. Convex order does not disappear from the picture, though. It still decides membership for densities at most one, and at density two it carries every single-step bound, every two-step bound, and every bound with nonincreasing gaps. What it cannot do is control laws with heavy atoms or fine spacing, and there the number of thresholds needed grows as the spacing shrinks, whatever the number of atoms. Three questions remain open, and they are not equally important.

The most important is \cref{prob:density}: do a density at most $2$ and $S\le_{\mathrm{cx}}2U$ together imply that $\law(S)$ is attainable? By \cref{prop:duality} it would be enough to settle densities with values in $\{0,2\}$ whose stop-loss gap is positive on $(0,2)$. Another route would be to write down, for every staircase, a convex certificate as in \cref{lem:densitydual}. Such certificates exist for single steps, for nonincreasing gaps, for two steps, and for every staircase with distinct thresholds on the grids of spacing $1/8$, $1/12$, and $1/16$. On this evidence the answer should be yes, but a construction that works uniformly is still missing.

The second question concerns atoms. \Cref{cor:smallatoms} needs every atom to be at most $1/k$, and the examples in Remark~S4.1 show that $1.01/k$ is already too much for some $k$. Does every bound $c/k$ with $c>1$ fail for some $k$? An answer would show exactly how far the lattice argument behind \cref{thm:density} can be pushed.

The third is about grids. A lattice law outside $\DU$ is excluded by some bound with repeated thresholds on its own grid, and by bounds with distinct thresholds only on a finer one (\cref{ex:repeated}). How much finer the grid has to be is open. On each fixed grid, the shortest complete list of bounds is the list of facets of $\Ph(k)$ other than the coordinate facets, and describing that list is the subject of the companion paper \cite{ChaudharyKhandelwalFacets}, which takes up the question of Li et al.\ \cite{LiLiuYao2026} on whether the rays behind these facets have a simple description.

\appendix
\section{Corrections to Mao et al.\ (2019)}\label{app:mww}

Numbering follows the journal version \cite{MaoWangWang2019}; arXiv version 2 numbers the same results Theorem~3.4, Theorem~3.5, and Lemma~3.7. Each item follows from \cref{cor:W3} or from direct arithmetic, and none affects their results for three or more summands. Items 1--3 are misprints in statements, item 4 concerns an interval endpoint and a construction in a lemma, and item 5 is a labeling slip inside a proof that does not affect the stated theorem.

\begin{enumerate}[leftmargin=*]
\item \emph{Theorem 4(iii).} The printed bound $f_2\ge a+b-\tfrac12$ should read $f_2\ge a+b-1$. The law $(1/18,1/18,8/9)$ on $\{7/18,13/18,19/18\}$ ($a=13/18$, $b=1/3$) belongs to $\DU$, by an exact construction, but violates the printed bound.
\item \emph{Summary after Theorem 4, parameter.} The printed $c=a/b+1-1/(2b)$ should read $c=(a+b-1)/b$. With the printed value the listed vertices leave the simplex (for $b=1/4$ and $a=0.8$ a coordinate equals $2.2$).
\item \emph{Summary, vertex of case (iii).} The printed $\tfrac12\bigl(c(1-b),\,2bc,\,2-c(1-b)\bigr)$ should read $\tfrac12\bigl(c(1-b),\,2bc,\,2-c(1+b)\bigr)$; the printed vector sums to $1+bc$. At $a=13/18$, $b=1/3$ the corrected vertex is exactly $(1/18,1/18,8/9)$.
\item \emph{Lemma 4(iii)(c).} The interval $p_1\in(b^*/T,b^*)$ should include its lower endpoint. It is attained by their construction once the last piece is corrected: on $\{X\in(2m,b+2m)\}$, where $T=2m+1$, the conditional law must be $Y=b-X-1$, not $Y=b-X$; with $Y=b-X$ the law of $Y$ is not uniform and that piece contributes nothing to $p_1$. This item was checked against the manuscript of Mao et al.\ dated June 8, 2019 (p.~14).
\item \emph{Proof of Theorems 3 and 4.} Under the labeling $P(Z_T=c-i)=f_{i+1}$ the mass $f_1$ sits at $a+b$, while the statement and the mean identity $f_3=f_1+(1-a)/b$ put $f_1$ at $a-b$. The middle mass, and hence the result, is unaffected.
\end{enumerate}
The bibliography entry of \cite{MaoWangWang2019} for Mao and Wang (2015) gives J.~Multivariate Anal.~136, 12--25; the paper is \cite{MaoWang2015}.

\section*{A note on the use of AI}
During the preparation of this work the authors used Claude Opus (Anthropic) and ChatGPT (OpenAI) with its Sol models in order to check proofs, write and verify the computation code, search the literature, and help with drafting the text. After using these tools, the authors reviewed and edited the content as needed and take full responsibility for the content of the published article.

\fund There are no funding bodies to thank relating to the creation of this article.

\competing There were no competing interests to declare which arose during the preparation or publication process of this article.

\data The code and data that check every computational statement of this paper are deposited at Zenodo, \href{https://doi.org/10.5281/zenodo.23231387}{doi:10.5281/zenodo.23231387}. Section~S5 of the supplementary material lists the checks result by result.

\ifjournal
\else
\supp The supplementary material (Sections S1 to S5: the general proof of \cref{thm:A}, the approximation rate, windows and four-atom laws, deferred proofs, and the exact checks) is posted with this article as the ancillary file supplement.pdf.
\fi

\bibliographystyle{APT}
\bibliography{refs}

\end{document}